\documentclass{amsart}

\usepackage{upgreek}
\usepackage{mathrsfs}
\usepackage{minibox}
\usepackage{amsmath}
\usepackage{amsthm}
\usepackage{amssymb}
\usepackage{amsfonts}
\usepackage[all]{xypic}
\usepackage{xcolor}
\usepackage{enumerate}
\usepackage{CJK}
\usepackage{bbm}

\begin{CJK}{UTF8}{}
\gdef\san{\mbox{\textbf{山}}}
\gdef\ten{\mbox{\textbf{天}}}
\end{CJK}

\renewcommand{\P}{\mathbb{P}}
\newcommand{\id}{\mathrm{id}}
\newcommand{\F}{\mathbb{F}}
\newcommand{\Frob}{\mathrm{Frob}}
\newcommand{\Q}{\mathbb{Q}}
\newcommand{\bep}{\mathbf{\upepsilon}}

\newcommand{\bmu}{\mathbf{\upmu}}
\newcommand{\Z}{\mathbb{Z}}
\newcommand{\C}{\mathbb{C}}

\DeclareMathOperator{\rank}{rank}
\DeclareMathOperator{\tr}{tr}

\DeclareMathOperator{\Out}{Out}
\DeclareMathOperator{\lcm}{lcm}
\DeclareMathOperator{\End}{End}
\DeclareMathOperator{\Hom}{Hom}

\DeclareMathOperator{\ord}{ord}
\DeclareMathOperator{\Gal}{Gal}
\DeclareMathOperator{\Br}{Br}

\newtheorem{theorem}{Theorem}
\newtheorem{lemma}[theorem]{Lemma}
\newtheorem{proposition}[theorem]{Proposition}
\newtheorem{corollary}[theorem]{Corollary}

\newtheorem{theorem*}{Theorem}
\newtheorem{lemma*}[theorem*]{Lemma}
\newtheorem{corollary*}[theorem*]{Corollary}

\newtheorem{conjecture}{Conjecture}

\theoremstyle{definition}
\newtheorem*{remark}{Remark}

\numberwithin{equation}{section}
\numberwithin{theorem}{section}

\title[Abelian varieties with constrained torsion]{Arithmetic of
  abelian varieties with constrained torsion}

\author[C.~Rasmussen]{Christopher Rasmussen}
\address{Wesleyan University, Middletown, Connecticut 06459, United States}
\email{crasmussen@wesleyan.edu}
\author[A.~Tamagawa]{Akio Tamagawa}
\address{Research Institute for Mathematical Sciences, Kyoto 606-8502, Japan}
\email{tamagawa@kurims.kyoto-u.ac.jp}

\begin{document}
\begin{CJK}{UTF8}{min}

\begin{abstract}
Let $K$ be a number field. We present several new finiteness results
for isomorphism classes of abelian varieties over $K$ whose
$\ell$-power torsion fields are arithmetically constrained for some
rational prime $\ell$. Such arithmetic constraints are related to an
unresolved question of Ihara regarding the kernel of the canonical
outer Galois representation on the pro-$\ell$ fundamental group of
$\P^1-\{0,1,\infty\}$.

Under GRH, we demonstrate the set of classes is finite for any fixed
$K$ and any fixed dimension. Without GRH, we prove a semistable
version of the result. In addition, several unconditional results are
obtained when the degree of $K/\Q$ and the dimension of abelian
varieties are not too large, through a careful analysis of the special
fiber of such abelian varieties. In some cases, the results (viewed as
a bound on the possible values of $\ell$) are uniform in the degree of
the extension $K/\Q$.
\end{abstract}

\maketitle

\section{Introduction}

\subsection{Introduction}

Let $\ell$ be a rational prime number, let $\bmu_N$ denote the $N$th
roots of unity, and $\bmu_{\ell^\infty} = \cup_{n\geq 1}
\bmu_{\ell^n}$. Set $\P^1_{01\infty} :=
\P^1_{\bar{\Q}}-\{0,1,\infty\}$. For a given number field $K$, we let
$\ten = \ten(K,\ell)$ denote the maximal pro-$\ell$ extension of
$K(\bmu_{\ell^\infty})$ which is unramified away from $\ell$. Let
$G_K$ denote the absolute Galois group $\Gal(\bar{K}/K)$, and consider
the natural outer Galois representation $\Phi \colon G_K \to \Out
\bigl( \pi_1^\ell ( \P^1_{01\infty}) \bigr)$. We let $\san =
\san(K,\ell)$ denote the subfield of $\bar{K}$ which is fixed by the
kernel of $\Phi$. Anderson and Ihara \cite{Anderson-Ihara:1988} have
shown that $\san$ is precisely the minimal field of definition
(containing $K$) of all curves appearing in the pro-$\ell$ tower of
(Galois) coverings of $\P^1$, branched only over
$\{0,1,\infty\}$. Moreover, they demonstrate many properties of
$\san$, including the containment $\san \subseteq \ten$.

Ihara has asked the following question (\cite{Ihara:1986}), which is
still open: For $K = \Q$, does $\san = \ten$? (The choice of the
notation is motivated as follows: the kanji $\ten$, read \emph{ten},
means ``heaven,'' and the kanji $\san$, read \emph{san}, means
``mountain.''  Both $\san$ and $\ten$ are infinite pro-$\ell$
extensions of $K(\bmu_\ell)$. Ihara's question is roughly as follows:
``Does the mountain reach the heavens?'')

There is a natural source for subextensions of $\ten$. Let $A$ denote
an abelian variety over $K$ which possesses good reduction away
from $\ell$. Then by the theory of Serre-Tate, the extension
$K(A[\ell^\infty])/K(A[\ell])$ is pro-$\ell$ and unramified away from
$\ell$; hence, in certain cases one finds $K(A[\ell^\infty]) \subseteq
\ten$. Reflecting on Ihara's question, it is natural to then study
whether or not $K(A[\ell^\infty]) \subseteq \san$. In several cases
where $A$ is the Jacobian variety of a curve $C$, this is known to
occur. For example, the containment holds for the following curves $C$
which appear in the pro-$\ell$ tower over $\P^1_{01\infty}$:
\begin{itemize}
\item Fermat curves and Heisenberg curves for any $\ell$
  \cite{Anderson-Ihara:1988},
\item Principal modular curves $X(2^n)$, $\ell = 2$
  \cite{Anderson-Ihara:1988},
\item Elliptic curves $E/\Q$, $\ell = 2$ \cite{Rasmussen:2004},
\item Elliptic curves $E/\Q$, $\ell = 3$
  \cite{Papanikolas-Rasmussen:2005}, 
\item Modular curves $X(3^n), X_0(3^n), X_1(3^n)$, $\ell = 3$
  \cite{Papanikolas-Rasmussen:2005}.
\end{itemize}
In addition, those elliptic curves $E/\Q$ with good reduction away
from $\ell$ and which have CM by $\Q(\sqrt{-\ell})$ are also known to
satisfy $K(E[\ell^\infty[) \subseteq \san$
\cite{Rasmussen-Tamagawa:2008}, although they do not lie in the
pro-$\ell$ tower over $\P^1_{01\infty}$.  

Despite the existence of these examples, the abelian varieties $A/K$
which satisfy $K(A[\ell^\infty]) \subseteq \ten$ appear to be quite
rare.  For an abelian variety $A/K$, let $[A]$ denote its
$K$-isomorphism class. For any fixed number field $K$, fixed integer
$g > 0$, and fixed rational prime number $\ell$, set
\begin{equation}
  \mathscr{A}(K,g,\ell) := \left\{ [A] : \dim A = g \mbox{ and }
    K(A[\ell^\infty]) \subseteq \ten \right\}. 
\end{equation}
For fixed $\ell$, this set is necessarily finite by the
Shafarevich Conjecture (Faltings' theorem). We also define
\begin{equation}
  \mathscr{A}(K,g) := \left\{ ([A], \ell) : [A] \in
    \mathscr{A}(K,g,\ell) \right\}.
\end{equation}
\begin{conjecture}\label{conj:AKg-finite}
  For any $K$ and $g$, the set $\mathscr{A}(K,g)$ is
  finite. Equivalently, the set $\mathscr{A}(K,g,\ell)$ is non-empty
  for only finitely many $\ell$.
\end{conjecture}
\begin{remark}
  If $A$ has everywhere good reduction, it is at least possible that
  $[A] \in \mathscr{A}(K,g,\ell)$ for more than one $\ell$. Hence, the
  reader should not assume that the natural surjection $\mathscr{A}(K,g)
  \to \bigcup_\ell \mathscr{A}(K,g,\ell)$, $([A], \ell) \mapsto [A]$,
  is a bijection.
\end{remark}
In \cite{Rasmussen-Tamagawa:2008}, the authors prove this conjecture
in the case $(K,1)$ for $K = \Q$ and for $K$ a quadratic extension of
$\Q$ other than the nine imaginary quadratic extensions of class
number one. Moreover, the set $\mathscr{A}(\Q,1)$ is determined
explicitly. It contains $50$ $\Q$-isomorphism classes, spanning $21$
$\Q$-isogeny classes. The containment related to Ihara's question,
$\Q(E[\ell^\infty]) \subseteq \san(\Q,\ell)$ is demonstrated for
almost all classes $([E], \ell) \in \mathscr{A}(\Q,1)$. There are $4$
isomorphism classes, spanning $2$ isogeny classes, which remain
open. In each of these classes, $\ell = 11$ and the representative
curve $E$ does \emph{not} have complex multiplication.

In the present article, we prove the finiteness of $\mathscr{A}(K,g)$,
for arbitrary $K$ and $g$, under the assumption of the Generalized
Riemann Hypothesis. In addition, several new cases of the conjecture
are proven unconditionally. When possible, we give proofs for uniform
versions of the conjecture, meaning we demonstrate the existence of a
constant $C$, possibly dependent on $g$ and $[K:\Q]$, but not $K$ itself,
so that $\ell > C$ implies $\mathscr{A}(K,g,\ell) = \varnothing$.

The organization of the paper is as follows. In \S\ref{sec:ANT},
several well-known results from analytic number theory are
collected. In \S3, the behavior of the Galois representation $\rho$ on
$A[\ell]$ is studied for any $[A] \in \mathscr{A}(K,g,\ell)$, leading
to constraints on the indices of semistable reduction. This yields a
proof of the conjecture when we restrict to abelian varieties with
semistable reduction. In \S4, we construct a character $\chi(m_\Q)$
from $\rho$, and demonstrate the remarkable property that $\chi(m_\Q)$
never vanishes on the Frobenius elements of small primes. This will
play a key role in the proofs of both the conditional and
unconditional finiteness results.

In \S5, we prove the conjecture under the assumption of the
Generalized Riemann Hypothesis in various forms. Actually, two proofs
are given. The first proves the finiteness of $\mathscr{A}(K,g)$ for
any choice of $(K,g)$. The second proves a version of the conjecture
which is uniform in the degree of $K/\Q$. Unfortunately, the second
proof requires the assumption that $[K:\Q]$ is odd. Finally, this
uniform result is generalized to the case of extensions of odd,
bounded degree of a fixed but arbitrary number field $F$.

The remainder of the paper is dedicated to unconditional proofs of the
conjecture for certain choices of $K$ and $g$. In \S6, the behavior of
the special fiber is used to further constrain the numerical
invariants introduced in \S3, \S4. These results are then used in \S7 to
prove the conjecture unconditionally in several new cases:
\begin{itemize}
\item $K = \Q$ and $g = 2$, $3$,
\item $[K:\Q] = 2$ and $g = 1$,
\item $[K:\Q] = 3$ and $g = 1$,
\item $K/\Q$ is a Galois extension of exponent $3$ and $g = 1$.
\end{itemize}
Moreover, in the case of cubic extensions and $g=1$, we are able to
give a uniform version of the result.

\subsection{Notations}
For any number field $F$, we let $\Delta_F$ denote the absolute
discriminant of $F/\Q$, and let $n_F = [F:\Q]$. For any extension of
number fields $E/F$, if $\frak{P}$ is a prime of $E$ above a prime
$\frak{p}$ of $F$, we let $e_{\frak{P}/\frak{p}}$ and
$f_{\frak{P}/\frak{p}}$ denote, respectively, the ramification index
and the degree of the residue field extension. We let
$\kappa(\frak{p})$ denote the residue field of $\frak{p}$.

Throughout, the notation $C_j = C_j(x,y,\dots,z)$ indicates a constant
$C_j$ which is dependent on $x, y, \dots, z$ and no other quantities.

\section{Ingredients from Analytic Number Theory}\label{sec:ANT}

In this section, we accumulate a few results from analytic
number theory that will be needed in the sequel. 

\subsection{Prime $m$-th power residues.}

Let $\ell$ be a prime number. Whenever $m \geq 1$ is a divisor of
$\ell - 1$, it will be useful to find a small rational prime $p$ which
is an $m$-th power residue modulo $\ell$; that is, for which $p\pmod{\ell}
\in \F_\ell^{\times m}$. Without further restriction on $m$, the best
known bound for $p$ is $p = O(\ell^{5.5})$, given by Heath-Brown in 
\cite{Heath-Brown:1992}. However, for $m < 23$, the following result
of Elliott gives a stronger bound
\cite{Elliott:1971}: 
\begin{proposition}\label{prop:Elliott}
  Let $m$ be a positive integer and $\varepsilon > 0$. There exists a
  constant $C_1' = C_1'(m, \varepsilon)$ such that for any prime
  $\ell$, there exists a prime $p < C_1' \cdot \ell^{\frac{m-1}{4} +
    \varepsilon}$ which is an $m$-th power residue modulo $\ell$. 
\end{proposition}
\begin{remark}
In fact, Elliott assumes $m \mid \ell - 1$. But note that in case $m
\nmid (\ell - 1)$, we have $\F_\ell^{\times m} = \F_\ell^{\times m'}$,
where $m' := \gcd(m, \ell - 1) < m$. So in fact the result holds as
stated. 
\end{remark}
We re-interpret Elliott's result as follows. For any integer $g > 0$
and any positive $\varepsilon < \frac{1}{4}$, set
\begin{equation*}
  C_1 = C_1(m, g, \varepsilon) := (4gC_1')^\frac{4}{(5-m)-4\varepsilon}. 
\end{equation*}
\begin{corollary}\label{cor:Elliott-corollary}
  Suppose $1 \leq m \leq 4$ and $0 < \varepsilon < \frac{1}{4}$. For any
  prime $\ell > C_1$, there exists a prime number $p <
  \frac{\ell}{4g}$ which is an $m$-th power residue modulo $\ell$.
\end{corollary}
\begin{proof}
  The quantity $\ell^{\frac{m-1}{4}+\varepsilon}$ is sub-linear in
  $\ell$; hence, there must be some lower bound for which such a $p$
  is guaranteed to exist. More precisely, one may check directly that $\ell > C_1$
  implies $C'_1 \ell^{\frac{m-1}{4} + \varepsilon} < \frac{\ell}{4g}$. 
\end{proof}

\subsection{Goldfeld's Theorem}

We recall a result of Goldfeld which will be used in the proof of
finiteness over quadratic fields when $g=1$. Let $K$ be a number
field, and let $S$ be a finite set of rational primes. Consider the
following two properties possibly satisfied by an integer $N$:
\begin{quote}
\begin{enumerate}
\item[({\bf Go 1})] There is a quadratic extension $L/\Q$ such that
  $-N$ is the discriminant of $L/\Q$. (Automatically, $L = \Q(\sqrt{-N})$.)
\item[({\bf Go 2})] If $p < \frac{|N|}{4}$ is a rational prime, $p
  \not\in S$, and $p$ splits completely in $K$, then $p$ does
  not split in $\Q(\sqrt{-N})$.
\end{enumerate}
\end{quote}
The following result of Goldfeld is proved in the 
Appendix of \cite{Mazur:1978}. (It is unfortunately not effective.)
\begin{theorem}[Theorem A, \cite{Mazur:1978}]\label{thm:Goldfeld}
Consider the set\begin{equation*}
\mathscr{N}(K,S) := \{N \in \Z : \mbox{\emph{$N$ satisfies both ({\bf
      Go 1}) and ({\bf Go 2})}}\,\}. 
\end{equation*}
If $n_K \leq 2$, then $\mathscr{N}(K,S)$ is finite.
\end{theorem}
We will rely on the following corollary to demonstrate
finiteness of $\mathscr{A}(K,1)$ for quadratic fields $K$.
\begin{corollary}\label{cor:Goldfeld}
  Suppose $n_K = 2$. There exists an ineffective constant $C_2 =
  C_2(K)$ such that, for any prime $\ell > C_2$, there exists a prime
  number $p < \frac{\ell}{4g}$ which is a square residue modulo $\ell$
  and which satisfies $f_{\frak{p}/p} = 1$ for any prime $\frak{p}$ of
  $K$ above $p$.
\end{corollary}
\begin{proof}
  For any odd prime $\ell$, let $\ell^* = (-1)^{\frac{\ell-1}{2}} \ell$,
  and notice that $\ell^*$ is the discriminant of
  $\Q(\sqrt{\ell^*})/\Q$. Further, $p$ splits in $\Q(\sqrt{\ell^*})$ if
  and only if $(\frac{p}{\ell}) = 1$. Let us set
\begin{equation*}
\mathscr{N}'(K) := \left\{ -\ell^* : \begin{tabular}{l} 
\mbox{$\ell$ odd prime such that for every $p < \frac{\ell}{4}$,} \\ 
\mbox{if $p$ splits in $K$ then $(\frac{p}{\ell})=-1$} 
\end{tabular} \right\}.
\end{equation*}
As $\mathscr{N}'(K) \subseteq \mathscr{N}(K,\varnothing)$, the result
follows immediately.
\end{proof}

\subsection{Chebotarev Density Theorem}

Let $E/F$ be a Galois extension of number fields, and let $\frak{p}$
be a prime of $F$, unramified in $E/F$. The Frobenius elements
of primes $\frak{P}$ of $E$ above $\frak{p}$ form a conjugacy class 
\begin{equation*}
\left[ \frac{E/F}{\frak{p}} \right] : = \{\Frob_\frak{P} : \frak{P}
\subseteq \mathscr{O}_E, \frak{P} \mid \frak{p} \}
\end{equation*}
inside $\Gal(E/F)$. 
If the particular choice of $\frak{P}$ is irrelevant, and no confusion
arises, we will write $\Frob_\frak{p}$ to denote any one element from
this class.

Let $\sigma \in \Gal(E,F)$. The Chebotarev
Density Theorem states that there are infinitely many primes
$\frak{p}$ of $F$, unramified in $E/F$, for which $\sigma \in \left[
  \frac{E/F}{\frak{p}} \right]$. We now recall an effective version of
this result, due to Lagarias and Odlyzko \cite{Lagarias-Odlyzko:1977},
conditional on the Generalized Riemann Hypothesis (GRH).

\begin{theorem}\label{thm:eff-Cheb}
  There exists an absolute constant $C_3 > 0$ with the following
  property. Let $E/F$ be a Galois extension of number fields, and
  suppose the Generalized Riemann Hypothesis holds for the Dedekind
  zeta function of $E$. For any $\sigma \in \Gal(E/F)$, there exists a
  prime $\frak{p}$ of $F$, unramified in $E/F$, with the following
  properties:
\begin{itemize}
\item $\sigma \in \left[ \frac{E/F}{\frak{p}} \right]$,
\item $N_{F/\Q}\frak{p} \leq C_3 (\log \Delta_E)^2$ (provided $E \neq \Q$).
\end{itemize}
\end{theorem}
\begin{remark}
This statement combines both Corollary 1.2 and the discussion on pages
461--462 of \cite{Lagarias-Odlyzko:1977}. 
\end{remark}
In exchange for a weakening of the bound on the norm, we may place an
additional constraint on $\frak{p}$.

\begin{corollary}\label{cor:eff-Cheb-f1}
Let $E/F$ be a Galois extension of number fields, and let $\tilde{E}$
denote the Galois closure of $E$ over $\Q$. Let $\sigma$ be a fixed
element of $\Gal(E/F)$. Assume GRH holds for the Dedekind zeta
function of $\tilde{E}$. Then there exists a prime $\frak{p}$ of $F$,
unramified in $E/F$, such that  
\begin{itemize}
\item $\sigma \in \left[ \frac{E/F}{\frak{p}} \right]$,
\item $N_{F/\Q} \frak{p} \leq C_3 (\log \Delta_{\tilde{E}})^2$
  (provided $E \neq \Q$),
\item $f_{\frak{p}/p} = 1$, where $p$ is the rational prime below
  $\frak{p}$.
\end{itemize}
\end{corollary}

\begin{proof}
  As $\sigma$ fixes $\Q$ and $\sigma(E) \subseteq \tilde{E}$, there
  exists $\tilde{\sigma} \in \Gal(\tilde{E}/\Q)$ such that
  $\tilde{\sigma} \bigr|_E = \sigma$.  Applying Theorem
  \ref{thm:eff-Cheb} to the extension $\tilde{E}/\Q$, we know there
  exists a rational prime $p < C_3 (\log \Delta_{\tilde{E}})^2$ such
  that $p$ is unramified in $\tilde{E}/\Q$ and $\tilde{\sigma} \in
  \left[ \frac{ \tilde{E}/\Q }{p} \right]$. Thus, there is a prime
  ideal $\frak{P} \mid p$ of $\tilde{E}$ such that $\tilde{\sigma} =
  \Frob_\frak{P}$. Necessarily, the decomposition group $D_\frak{P}
  \leq \Gal(\tilde{E}/\Q)$ is generated by $\tilde{\sigma}$. As
  $\tilde{\sigma}$ fixes $F$, we in fact have $D_\frak{P} \leq
  \Gal(\tilde{E}/F)$. Let $F_1$ denote the subextension of
  $\tilde{E}/F$ fixed by $D_\frak{P}$, and set $\frak{p}_1 := \frak{P}
  \cap \mathscr{O}_{F_1}$. Necessarily, the residue fields
  $\mathscr{O}_{F_1}/\frak{p}_1$ and $\Z/p\Z$ coincide, and so
  $f_{\frak{p}_1/p} = 1$. Setting $\frak{p} = \frak{p}_1 \cap
  \mathscr{O}_F$, we have $f_{\frak{p}/p} = 1$, also. Moreover,
  $\sigma \in \left[ \frac{E/F}{\frak{p}} \right]$; since
  $N_{F/\Q}\frak{p} = p$, the result is shown.
\end{proof}

Suppose $\ell$ is a rational prime and $m \mid (\ell
- 1)$. We let $\Q(\bmu_\ell)_m$ denote the unique subfield of
$\Q(\bmu_\ell)$ which is a degree $m$ extension of $\Q$. 

\begin{proposition}\label{prop:GRH-Cheb-bound}
Let $m \geq 1$, $g > 0$, and $n \geq 1$ be fixed integers. Let $K$
be a fixed number field, with Galois closure $\tilde{K}$ over
$\Q$. There exists a constant $C_6 = C_6(m, g, n, K)$ with the
following property. Suppose $\ell > C_6$ is a prime number, set $L_0 =
\Q(\bmu_\ell)_m$, and suppose the Generalized Riemann Hypothesis holds
for the Dedekind zeta function of $L_0 \tilde{K}$. Then for any
$\sigma \in \Gal(L_0 \tilde{K} / K)$, there exists a rational prime $p
< \left( \frac{\ell}{4g} \right)^{1/n}$ and a prime $\frak{p} \mid p$
of $K$, such that
\begin{itemize}
\item $\frak{p}$ is unramified in $L_0\tilde{K} / K$,
\item $f_{\frak{p}/p} = 1$,
\item $\sigma \in \left[ \frac{L_0 \tilde{K} / K}{\frak{p}} \right]$.
\end{itemize}
Consequently, for $\ell > C_6$, there exists a prime $p < \left(
  \frac{\ell}{4g} \right)^{1/n}$ which is an $m$-th power residue modulo $\ell$.
\end{proposition}
\begin{proof}
  Since $C_6$ may depend on $K$, we may assume $\ell \nmid \Delta_K$
  without loss of generality. Let $L = L_0 K$, $\tilde{L} = L_0
  \tilde{K}$. Then $\tilde{L}$ is the Galois closure of $L$ over
  $\Q$. The fields $L_0$ and $\tilde{K}$ are linearly disjoint over
  $\Q$ and have no common factor in their discriminants. So
  (\cite[Prop.~17]{Lang:1994}) we have 
\begin{equation*}
  \Delta_{\tilde{L}} = \Delta_{L_0 \tilde{K}} =
  \Delta_{\tilde{K}}^{n_{L_0}} \cdot \Delta_{L_0}^{n_{\tilde{K}}} =
  \Delta_{\tilde{K}}^m \cdot (\ell^{m-1})^{n_{\tilde{K}}}. 
\end{equation*}
Consequently, we always have the bound
\begin{equation*}
\begin{split}
\log \Delta_{\tilde{L}} & = m \log \Delta_{\tilde{K}} +
(m-1)n_{\tilde{K}} \log \ell \\
& \leq m \log \Delta_{\tilde{K}} + (m-1)n_K! \log \ell.
\end{split}
\end{equation*}
If $L=\Q$ (i.e., if $m=1$ and $K=\Q$), the assertion clearly holds 
with any $C_6> 4g\cdot 2^n$, since we may then take $p=2$. So we may 
assume $L \neq \Q$. Combining with Corollary \ref{cor:eff-Cheb-f1}, we
see there exists a rational prime $p$ and a prime $\frak{p} \mid p$ of
$K$, such that $\frak{p}$ is unramified in $\tilde{L}$,
$f_{\frak{p}/p} = 1$, and $\sigma \in \left[
  \frac{\tilde{L}/K}{\frak{p}} \right]$. Moreover, $p$ may be chosen 
so that 
\begin{equation}
p \leq C_3 \cdot (C_4 + C_5 \log \ell)^2,
\end{equation}
where
\begin{equation*}
\begin{split}
C_4 = C_4(m, K) & := m \log \Delta_{\tilde{K}}, \\
C_5 = C_5(m, n_K) & := \max \{1, (m-1) n_K! \}.
\end{split}
\end{equation*}
As $C_3 \cdot (C_4 + C_5 \log \ell)^2 < \left( \frac{\ell}{4g} \right)^{1/n}$ 
for $\ell \gg 0$, this proves the first claim. For the second claim, choose $\sigma \in
\Gal(\tilde{L}/K)$ such that $\sigma|_{L_0} = \mathrm{id}$. Then the
prime $\frak{p}$ guaranteed by the first claim has an associated
Frobenius element which is trivial on $L_0$; this implies that $p$ is
an $m$-th power residue modulo $\ell$.
\end{proof}
\begin{remark}
Notice that this result generalizes (in fact, implies, under GRH), the
earlier results of the section which guarantee a small prime $m$-th
power residue.
\end{remark}
\begin{remark}
  Theorem \ref{thm:eff-Cheb} remains valid even if $C_3$ is replaced
  by a larger constant. So we may and do assume $C_3 \geq 1$. Let
  $\ell'$ denote the largest prime divisor of $\Delta_K$. For the
  constant $C_6$, we may take the value (provided $(m,K) \neq (1,\Q)$)
\begin{equation*}
  C_6(m,g,n,K) := \max \left\{\ell', 16g^2C_3^{2n}C_5^{4n}(2n)^{4n} \exp \left( \frac{C_4}{C_5} \right) \right\}.
\end{equation*}
This follows from a lengthy argument that when $\ell > C_6$, $\ell$
also satisfies the inequality
\begin{equation*}
\left( \frac{\ell}{4g} \right)^{1/n} > C_3 \cdot (C_4 + C_5 \log \ell)^2.
\end{equation*}
The details of the argument are given in the appendix.
\end{remark}

\section{Constraints on the indices of semistable reduction}

Let $K$ be a number field, and $A/K$ an abelian variety of dimension $g
> 0$. Let $\ell$ be a rational prime. For any prime $\lambda$ of $K$
above $\ell$, denote by $K_\lambda$ the $\lambda$-adic completion of
$K$. Let $A_{K_\lambda}$ denote the base change of $A$ over
$K_\lambda$, and let $e_{A_{K_\lambda}}$ be the minimal ramification
index at $\lambda$ for which semistable reduction for $A_{K_\lambda}$
is achieved. In this section, we record some constraints on
$e_{A_{K_\lambda}}$ in general, and also under the assumption that
$[A] \in \mathscr{A}(K,g,\ell)$.

\subsection{The index of semistable ramification.}

Let $K_\lambda^\mathrm{ur}$ denote the maximal unramified extension of
$K_\lambda$, and let $I_\lambda = G_{K_\lambda^\mathrm{ur}} \subset G_{K_\lambda}$
denote the inertia group at $\lambda$. The ramification index $e_{\lambda/\ell}$
divides $[K_\lambda : \Q_\ell] \leq n_K$. Moreover,
$e_{\lambda/\ell} > 1$ for some $\lambda \mid \ell$ if and only if
$\ell \mid \Delta_K$. It is known that, for any prime $\ell' \neq
\ell$, the kernel $J_\lambda$ of the natural representation
$\rho_{A,\ell'}^\mathrm{ss} \colon I_\lambda \to
\mathrm{GL}(V_{\ell'}(A)^\mathrm{ss})$ is an open subgroup of
$I_\lambda$ independent of the choice of $\ell'$. Further, $J_\lambda$
has index $e_{A_{K_\lambda}}$ in $I_\lambda$, so we may be sure that
$e_{A_{K_\lambda}} \mid \#\mathrm{GL}_{2g}(\F_{\ell'})$ \emph{for
  every $\ell' \nmid 2\ell$.}  (These are consequences of
\cite[Expos\'e IX]{SGA7I}.) Hence, the following is useful:   
\begin{lemma}
Fix an integer $n > 0$. For any prime $p$ and any odd prime $\ell'$,
the $p$-part of $\#\mathrm{GL}_n(\F_{\ell'})$ is divisible by
$p^{u_p}$, where 
\begin{equation*}
\begin{split}
u_2 & := v_2(n!) + n + \lfloor \tfrac{n}{2} \rfloor, \\
u_p & := v_p \left( \left\lfloor \tfrac{n}{p-1} \right\rfloor ! \right) +
\left\lfloor \tfrac{n}{p-1} \right\rfloor \qquad (p > 2).
\end{split}
\end{equation*}
Here, $v_p$ denotes the $p$-adic valuation, and $\lfloor \cdot
\rfloor$ denotes the greatest integer function. Moreover, for any $p$,
there are infinitely many $\ell'$ such that the $p$-part of
$\#\mathrm{GL}_n(\F_{\ell'})$ is exactly $p^{u_p}$.
\end{lemma}
\begin{proof}
The result is not new. For the case of odd $p$, a proof is given in
\cite[Lemma 7]{Guralnick-Lorenz:2006}; for the even case, a similar
argument can be constructed by considering primes $\ell' \equiv 3
\pmod{8}$. The formulas given in \cite[pg.~120]{Serre:1979}
are helpful. 
\end{proof}
Note, in particular, that $u_p = 0$ for $p > n + 1$. Consequently, the
product
\begin{equation*}
M'(n) := \prod_{\mbox{$p$ prime}} p^{u_p}
\end{equation*}
is always finite, and gives the greatest common divisor of $\{
\#\mathrm{GL}_n(\F_{\ell'}) : \ell' \nmid 2\ell \}$. The notation
$M'(n)$ is inspired by the similarity to the quantity $M(n)$, which
gives the least common multiple of all the orders of finite subgroups
of $\mathrm{GL}_n(\Q)$. This was first computed by Minkowski
\cite{Minkowski:1887} -- see \cite{Guralnick-Lorenz:2006} for a modern
account. In any case, we obtain the following:   
\begin{corollary}\label{cor:e-small-prime-factors}
The index $e_{A_{K_\lambda}}$ divides $M'(2g)$. Moreover, if $p$ is a
prime with $p \mid e_{A_{K_\lambda}}$, then $p \leq 2g + 1$.
\end{corollary}

\subsection{Structure of $G_K$-action on $\ell$-torsion}

For the remainder of this section, we always work under the following 
assumption: 
\begin{equation}
\boxed{[A] \in \mathscr{A}(K,g,\ell)} \tag{A1}\label{A1}
\end{equation}
Let $\chi \colon G_\Q \to \F_\ell^\times$ denote the cyclotomic
character modulo $\ell$. Set $\delta := [\F_\ell^\times : \chi(G_K)]$,
and note that $\delta$ divides both $n_K$ and $\ell-1 =
\#\F_\ell^\times$. We let $\rho_{A,\ell}$ denote the representation of
$G_K$-action on $A[\ell]$. If $[A] \in \mathscr{A}(K,g,\ell)$, the
abelian variety $A$ must have good reduction away from
$\ell$. Moreover, the structure of $\rho_{A,\ell}$ is constrained as
follows: 
\begin{lemma}\label{structure_lemma}
Under \eqref{A1}, there is a basis of $A[\ell]$ with respect to which
\begin{equation*}
\rho_{A,\ell} = \left( \begin{array}{cccc}
\chi^{i_1} & * & \cdots & * \\
           & \chi^{i_2} & \cdots & * \\
 & & \ddots & \vdots \\
 & & & \chi^{i_{2g}}
\end{array}
\right).
\end{equation*}
Moreover, the indices $i_r$ may be chosen so that $i_r \in \Z \cap [0,
\frac{\ell - 1}{\delta})$ for all $1 \leq r \leq 2g$.
\end{lemma}
\begin{proof}
This is an obvious generalization of \cite[Lemma
3]{Rasmussen-Tamagawa:2008}, and in fact, the proof given there may be
followed almost verbatim, with $G = \Gal(\ten/K)$,
$N=\Gal(\ten/K(\bmu_\ell))$, $\Delta = \Gal(K(\bmu_\ell)/K)$. In
\cite{Rasmussen-Tamagawa:2008}, it is only claimed that $i_r < \ell -
1$. We may be sure that the stronger bounds on $i_r$ hold, by the
following lemma.  
\end{proof}

\begin{lemma}
Let $G$ be a profinite group, and let $N \lhd G$ be a normal
pro-$\ell$ subgroup of $G$. Suppose that $\Delta$ is a finite cyclic
group and $\ell \nmid \#\Delta$. Let $\chi \colon G \to \Delta$ be a
group homomorphism with $\ker \chi = N$, and let $\psi \colon G \to
\Delta$ be any other group homomorphism. Then there exists $b \in \Z
\cap [0, \#\chi(G))$ such that $\psi = \chi^b$. 
\end{lemma}
\begin{proof}
As $\ell \nmid \#\Delta$, we clearly have $N \leq \ker \psi$. Now,
both $G/N$ and $\chi(G)$ are cyclic, so let $gN$ and $x$ be generators
of these respective groups such that $\chi(g) = x$. By the
containments $N \leq \ker \psi \leq G$, we see $\#\psi(G) = [G: \ker
\psi]$ divides $[G:N] = \#\chi(G)$. We must have $\psi(G) \leq
\chi(G)$, since these are subgroups of the
same cyclic group $\Delta$. So $\psi(g) = x^b$ for some $0 \leq b <
\#\chi(G)$. Necessarily, $\psi = \chi^b$. 
\end{proof}

\subsection{Tate-Oort Theory}

Let $L$ be the Galois extension of $K_\lambda^\mathrm{ur}$ of degree
$e_{A_{K_\lambda}}$ corresponding to $J_\lambda$. Note that the extension
$L/K_\lambda^\mathrm{ur}$ descends (non-canonically) to a (possibly
non-Galois) extension of $K_\lambda$ of degree $e_{A_{K_\lambda}}$,
and even descends to an extension of $K$ of degree $e_{A_{K_\lambda}}$
(by approximation). As $A_L$ is semistable over $\mathscr{O}_L$, by
\cite[Expos\'e IX, Prop.~5.6]{SGA7I}, we see that each character
$\chi^{i_r} \colon J_\lambda = G_L \to \F_\ell^\times$ extends to a
finite group scheme over $\mathscr{O}_L$. Let $\psi_\lambda \colon
J_\lambda \rightarrow \F_\ell^\times$ be the fundamental character
over $L$. Thus, $\chi = \psi_\lambda^{e_\lambda}$ on $J_\lambda$,
where $e_\lambda := e_{A_{K_\lambda}} \cdot e_{\lambda/\ell}$, by
\cite[\S1, Prop.~8]{Serre:1972}.   

By the theory of Tate-Oort \cite{Tate-Oort:1970}, $\chi^{i_r} =
\psi_\lambda^{j_{\lambda,r}}$ on $J_\lambda$, where $j_{\lambda,r} \in
\Z \cap [0, e_\lambda]$. From this, we obtain:
\begin{equation}\label{eq:eij}
e_\lambda i_r \equiv j_{\lambda,r} \pmod{\ell - 1}.
\end{equation}
Among all primes of $K$ which do not divide $\ell$, choose
$\frak{p}_0$ whose residue field $\kappa(\frak{p}_0)$ is of minimal
order, and set $q_0 := \#\kappa(\frak{p}_0)$. Considering primes of
$K$ above $2$ and $3$, we see $q_0 \leq 3^{n_K}$ in general and $q_0
\leq 2^{n_K}$ if $\ell \neq 2$. For any integer $n > 0$, let
$P_{\frak{p}_0, n} \in \Z[T]$ denote the characteristic polynomial of
$\Frob_{\frak{p}_0}^n$ acting on $V_\ell(A)$, which has degree
$2g$. Fix an algebraic closure of $\Q$ and let
$\{\alpha_{\frak{p}_0,r}\}_{r=1}^{2g}$ denote the roots of
$P_{\frak{p}_0, 1}$ (counting multiplicity). These roots satisfy
$|\alpha_{\frak{p}_0,r}| = q_0^{1/2}$,
and the $n$-th powers of the $\alpha_{\frak{p}_0,r}$ give exactly the
roots of $P_{\frak{p}_0,n}$. On the other hand, modulo $\ell$, the
roots of $P_{\frak{p}_0,n}$ are given by $\{
\chi^{i_r}(\Frob_{\frak{p}_0}^n) \}_{r=1}^{2g}$. From this and the
congruence \eqref{eq:eij}, we have
\begin{equation}\label{eq:alpha-q}
\prod_{r=1}^{2g} \bigl( T - \alpha_{\frak{p}_0,r}^{e_\lambda}) \bigr)
= P_{\frak{p}_0, e_\lambda}(T) \equiv \prod_{r = 1}^{2g} (T -
q_0^{j_{\lambda,r}}) \pmod{\ell}.  
\end{equation}
Let
\begin{equation*}
S(T,x_1,\dots,x_{2g}) := \prod_{i=1}^{2g} (T - x_i) \in
\Z[x_1,\dots,x_{2g}][T], 
\end{equation*}
and let $S_k(x_1, \dots, x_{2g})$ denote the coefficient of $T^{2g-k}$
in $S$. The polynomials $S_k$ are symmetric in the $x_j$, and so
$S_k(\alpha_{\frak{p}_0,1}^n, \dots, \alpha_{\frak{p}_0,2g}^n) \in \Z$
for any $n \geq 1$. Using \eqref{eq:alpha-q}, we have
\begin{equation}\label{eq:Sm_cong}
S_k(\alpha_{\frak{p}_0,1}^{e_\lambda}, \dots, \alpha_{\frak{p}_0,2g}^{e_\lambda}) \equiv
S_k(q_0^{j_{\lambda,1}}, \dots, q_0^{j_{\lambda,2g}}) \pmod{\ell}.
\end{equation}
As $S_k$ is a homogeneous polynomial of degree $k$ with
$\binom{2g}{k}$ terms, and $j_{\lambda,r} \leq e_\lambda$, we
certainly have by the triangle inequality:
\begin{equation*}
\left| S_k(\alpha_{\frak{p}_0,1}^{e_\lambda}, \dots,
  \alpha_{\frak{p}_0,2g}^{e_\lambda}) - S_k(q_0^{j_{\lambda,1}},
  \dots, q_0^{j_{\lambda,2g}}) \right| \leq \binom{2g}{k} \cdot
q_0^{e_\lambda k/ 2} + \binom{2g}{k} q_0^{e_\lambda k}.
\end{equation*}
We add the following assumption:
\begin{equation}
\boxed{\quad \ell > \max \left\{ \binom{2g}{k}
    \left( q_0^{e_\lambda k} + q_0^{\frac{e_\lambda k}{2}} \right) :
    \lambda \mid \ell,\; 1 \leq k \leq 2g \right\}. \quad} \tag{A2} \label{A2}
\end{equation} 
For example, this is certainly satisfied if
\begin{equation}\label{eq:A2Bound}
\ell > C_7 = C_7(g, n_K) := 2 \cdot {2g \choose g} \cdot 3^{2g \cdot
  n_K^2 \cdot M'(2g)}. 
\end{equation}
Take $k = 1$; as $q_0 \geq 2$ and $e_\lambda \geq 1$, we note that
$\ell > 2g + 1$ always under \eqref{A2}.

Under \eqref{A2}, the congruences \eqref{eq:Sm_cong} require equality in
$\Z$, which means the sets (possibly with multiplicity) $\{
\alpha_{\frak{p}_0,r}^{e_\lambda} \}_{r=1}^{2g}$ and $\{
q_0^{j_{\lambda,r}} \}_{r=1}^{2g}$ must be equal. By the Weil
conjectures, we must have $j_{\lambda,r} = \frac{1}{2}e_\lambda$ for
each $r$. Thus, $2 \mid e_\lambda$. Moreover, $A_L$ has good reduction
with $\ell$-rank $0$. (Otherwise, we would have $j_{\lambda,r} = 
e_\lambda$ for some $r$.) Combining with \eqref{eq:eij}, we obtain: 
\begin{equation}\label{eq:e-lambda}
e_\lambda i_r \equiv \frac{e_\lambda}{2} \pmod{(\ell-1)}, \qquad 1 \leq r \leq
2g.
\end{equation}
Set $e = \gcd \{ e_\lambda : \lambda \mid \ell\}$.
\begin{lemma}\label{lem:e-properties}
Assume \eqref{A1} and \eqref{A2} hold. Then
\begin{enumerate}[\quad (a)]
\item $e \mid M'(2g)n_K$,
\item $e \mid M'(2g)$ if $\ell \nmid \Delta_K$,
\item $(e, \ell - 1) = (\frac{e}{2}, \ell - 1)$,
\item $4 \mid e$,
\item For any $1 \leq r, s \leq 2g$, $\frac{e}{2}(i_r + i_s - 1)
  \equiv 0 \pmod{(\ell - 1)}$.
\end{enumerate}
\end{lemma}
\begin{proof}
Let $n_\lambda$ denote the local degree $[K_\lambda : \Q_\ell]$ at
$\lambda$. As $\sum_{\lambda \mid \ell} n_\lambda = n_K$ and
$e_{\lambda/\ell} \mid n_\lambda$, we see that $\gcd \{e_{\lambda/\ell}
: \lambda \mid \ell\} \bigm| n_K$. Now, by Corollary
\ref{cor:e-small-prime-factors},
\begin{equation*}
e = \gcd_{\lambda \mid \ell} \{e_{A_{K_\lambda}} \cdot e_{\lambda/\ell} \}
  \Bigm| \gcd_{\lambda \mid \ell} \{M'(2g) \cdot e_{\lambda/\ell} \} =
  M'(2g) \gcd_{\lambda \mid \ell} \{e_{\lambda/\ell} \} \Bigm| M'(2g)n_K,
\end{equation*}
which proves (a). When $\ell \nmid \Delta_K$, all $e_{\lambda/\ell} =
1$, so that $e = \gcd \{e_{A_{K_\lambda}} : \lambda \mid \ell\}$, and
(b) follows by Corollary \ref{cor:e-small-prime-factors} also. Since
$2 \mid e_\lambda$ for all $\lambda$, $e$ must be even. Now, from
\eqref{eq:e-lambda}, we deduce
\begin{equation}\label{eq:e2ir-1}
  \frac{e}{2}(2i_r - 1) \equiv 0 \pmod{(\ell-1)}.
\end{equation}
As $(2i_r - 1)$ is odd, this implies $\ord_2(e) >
\ord_2(\ell-1)$. Thus, (c) holds. Under \eqref{A2}, $\ell > 2$, so
$\ord_2(e) > 1$, proving (d). Finally, adding the congruence
\eqref{eq:e2ir-1} for two indices $1 \leq r,s \leq 2g$ gives
\begin{equation*}
e(i_r + i_s - 1) \equiv 0 \pmod{(\ell-1)}.
\end{equation*}
This, combined with (c), implies (e).
\end{proof}

Already, we may prove a finiteness result for everywhere semistable
abelian varieties. For fixed $K, g, \ell$, let
$\mathscr{A}^\mathrm{ss}(K,g,\ell)$ denote the subset of
$\mathscr{A}(K,g,\ell)$ containing only classes of abelian varieties 
with everywhere semistable reduction; likewise, let
$\mathscr{A}^\mathrm{ss}(K,g)$ denote the set of pairs $([A], \ell)
\in \mathscr{A}(K,g)$ for which $A$ has everywhere semistable
reduction.
\begin{theorem}\label{thm:semistable-finite}
For any $K$ and any $g > 0$, the set $\mathscr{A}^\mathrm{ss}(K,g)$
is finite. Equivalently, $\mathscr{A}^\mathrm{ss}(K,g,\ell) =
\varnothing$ for $\ell \gg 0$. 
\end{theorem}
\begin{proof}
  For sufficiently large $\ell$, we may be sure that both $\ell \nmid
  \Delta_K$, and that \eqref{A2} holds. Suppose $[A] \in
  \mathscr{A}^\mathrm{ss}(K,g,\ell)$. As $\ell \nmid \Delta_K$, we
  know $e_{\lambda/\ell} = 1$ for every prime $\lambda \mid \ell$ in
  $K$. Hence, $e_\lambda = e_{A_{K_\lambda}}$ for every $\lambda$. But
  as $A$ is already semistable at $\lambda$, $e_{A_{K_\lambda}} =
  1$. Thus $e_\lambda = 1$, and so $e = 1$ also. But under \eqref{A2},
  $e > 1$, a contradiction. Thus, $\mathscr{A}^\mathrm{ss}(K,g,\ell) =
  \varnothing$.
\end{proof}
In fact, a uniform version of Theorem \ref{thm:semistable-finite} is
available for many values of $n_K$.
\begin{lemma}\label{lem:2-eAKlambda}
Suppose $[A] \in \mathscr{A}(K,g,\ell)$ and $\ell > C_7(g,n_K)$.
\begin{enumerate}[(a)]
\item If $2 \nmid e_{A_{K_\lambda}}$ for every $\lambda \mid \ell$,
  then $4 \mid n_K$.
\item If $4 \nmid e_{A_{K_\lambda}}$ for every $\lambda \mid \ell$,
  then $2 \mid n_K$.
\end{enumerate}
\end{lemma}
\begin{proof}
  By Lemma \ref{lem:e-properties}, we know $4 \mid e_{A_{K_\lambda}}
  e_{\lambda/\ell}$ for each $\lambda \mid \ell$. If $2 \nmid
  e_{A_{K_\lambda}}$ for every $\lambda$, then we must have $4 \mid
  e_{\lambda/\ell}$. Since $n_K = \sum_{\lambda \mid \ell}
  e_{\lambda/\ell} f_{\lambda/\ell}$, we obtain (a). Part (b) may be
  argued the same way.
\end{proof}
Thus, we obtain a uniform version of Theorem
\ref{thm:semistable-finite} for many values of $n_K$.
\begin{corollary}
  Let $n$ be a positive integer, not divisible by $4$. For any number
  field $K/\Q$ with $n_K = n$, any integer $g > 0$, and any rational
  prime $\ell > C_7(g,n)$, $\mathscr{A}^\mathrm{ss}(K,g,\ell) =
  \varnothing$.
\end{corollary}
\begin{proof}
  Were $\mathscr{A}^\mathrm{ss}(K,g,\ell)$ non-empty, it would contain
  a class $[A]$ for which $e_{A_{K_\lambda}} = 1$ for every $\lambda
  \mid \ell$. However, this contradicts Lemma \ref{lem:2-eAKlambda}(a).
\end{proof}
\begin{remark}
  Note that the proof of Theorem \ref{thm:semistable-finite} actually
  yields a stronger result, as we only need the existence of one
  $\lambda \mid \ell$ for which $A$ possesses semistable
  reduction. Hence, we have actually proven the finiteness of the
  subset of pairs $([A], \ell)$ in $\mathscr{A}(K,g)$ for which $A$ possesses
  semistable reduction for at least one prime of $K$ dividing
  $\ell$. (To be clear, this improvement is not available in the
  uniform version of the corollary, which requires semistable
  reduction at every prime above $\ell$.)
\end{remark}

\section{Supersingularity at small primes}

\subsection{The homomorphism $\bep$} We keep the notations of the
previous section, and assume the hypotheses \eqref{A1} and \eqref{A2}
hold. Recall $\chi$ denotes the cyclotomic character
modulo $\ell$. For any $r$ and $s$ with $1 \leq r, s \leq 2g$, set
$\varepsilon_{r,s} := \chi^{i_r + i_s - 1}$. We further define
\begin{equation*}
\begin{split}
\bep := (\varepsilon_{r,s})_{1 \leq r, s \leq 2g} & \colon G_\Q
\longrightarrow (\F_\ell^\times)^{(2g)^2}, \\
\bep_0 := (\varepsilon_{r,r})_{1 \leq r \leq 2g} & \colon G_\Q
\longrightarrow (\F_\ell^\times)^{2g}.
\end{split}
\end{equation*}
Set $m_\Q := \# \bep(G_\Q)$, and $m_{0,\Q} := \#
\bep_0(G_\Q)$. Then $m_\Q$ is the least common multiple of the orders
of the $\varepsilon_{r,s}$, and $m_{0,\Q}$ is likewise the least
common multiple of the orders of the $\varepsilon_{r,r}$. Hence,
$m_{0,\Q} \mid m_\Q$. Clearly $\bep$ factors through $\F_\ell^\times$,
so $m_\Q \mid (\ell - 1)$ and $\bep(G_\Q)$ is cyclic. Further, the image
has exponent $\frac{e}{2}$, by Lemma \ref{lem:e-properties}(e). Thus,
$m_\Q \mid \frac{e}{2}$, and so $m_{0,\Q} \mid m_\Q \mid \left(
  \tfrac{e}{2}, \ell - 1 \right)$.
\begin{lemma}\label{lemma:mQ}
We have $m_{0,\Q} = m_\Q$. Moreover, $\ord_2 m_\Q = \ord_2 (\ell -
1)$. In particular, $2 \mid m_\Q$. 
\end{lemma}
\begin{proof}
For any $r$ and $s$, note that $\varepsilon_{r,r} \cdot
\varepsilon_{s,s} = \varepsilon_{r,s}^2$. Thus, $m_\Q \mid 2m_{0,\Q}$. We
certainly have: 
\begin{equation*}
\ord_2(m_{0,\Q}) \leq \ord_2 (m_\Q) \leq \ord_2 \left( \tfrac{e}{2},
  \ell - 1 \right) \leq \ord_2 (\ell - 1).
\end{equation*}
Since $\chi^{(2i_r - 1)m_{0,\Q}} = \varepsilon_{r,r}^{m_{0,\Q}} = 1$, we have
$(2i_r - 1)m_{0,\Q} \equiv 0 \pmod{\ell - 1}$. Since $2i_r - 1$ is odd, it
follows that $\ord_2 m_{0,\Q} \geq \ord_2 (\ell - 1)$, and so all four terms
in the inequality are equal. This implies $m_\Q \mid m_{0,\Q}$, and so
$m_\Q = m_{0,\Q}$. 
\end{proof}

\subsection{The characters $\chi(m)$.}

Let $m$ be an integer dividing $\ell - 1$. We let $\chi(m)\colon G_\Q
\to \F_\ell^\times / \F_\ell^{\times m}$ denote the character $\chi$
modulo $m$-th powers. Then the character $\chi(m_\Q)$ carries essentially the
same information as the homomorphism $\bep$. Indeed, we have just seen
that $\bep$ factors through $\F_\ell^\times$. Since the image is
cyclic of order $m_\Q$, we have an isomorphism $\F_\ell^\times /
\F_\ell^{\times m_\Q} \cong \bep(G_\Q)$, and the following diagram commutes:
\begin{equation*}
\xymatrix @C=1cm {
& \bep(G_\Q) & \\
G_\Q \ar[ur]^{\bep} \ar[r]^\chi \ar@/_4mm/@{-->}[rr]_{\chi(m_\Q)} & \F_\ell^\times
\ar@{->>}[u] \ar@{->>}[r] & \F_\ell^\times / \F_\ell^{\times m_\Q}
\ar@{->>}[lu]^\cong }
\end{equation*}
Let $p \neq \ell$ be a rational prime. Since $\chi(\Frob_p) \equiv p
\pmod{\ell}$, we see that $\bep(\Frob_p)$ is trivial precisely when
$p$ is an $m_\Q$-th power residue modulo $\ell$. 

\subsection{A Technique of Mazur}

In \cite[\S7]{Mazur:1978}, Mazur deduces congruences from the
existence of an isogeny of elliptic curves. Here, we follow the spirit
of Mazur's idea and use it to study the behavior of $\chi(m_\Q)$. Fix
a prime number $\ell$, a number field $K$, and $g > 0$. Suppose $A/K$
is an abelian variety for which \eqref{A1} and \eqref{A2} hold. Let
$p$ be a rational prime, and suppose $\frak{p}$ is a prime of $K$
which divides $p$. We let $q = N_{K/\Q}\frak{p} =
p^{f_{\frak{p}/p}}$. Finally, let $\Frob_\frak{p} \in G_K$ denote a
Frobenius element associated to $\frak{p}$.

\begin{proposition}\label{prop:Mazur_trick}
Suppose $f_{\frak{p}/p}$ is odd, and $q < \frac{\ell}{4g}$. Then
$\chi(m_\Q)(\Frob_\frak{p}) \neq 1$.
\end{proposition}
\begin{proof}
  For the sake of contradiction, suppose $\chi(m_\Q)(\Frob_\frak{p}) =
  1$; in particular, this forces $\varepsilon_{r,s}(\Frob_\frak{p}) =
  1$ for all $r$, $s$. Let $\{ \alpha_{\frak{p},r} \}_{r=1}^{2g}$ be
  the eigenvalues for $\Frob_\frak{p}$. Note that $\{
  \alpha_{\frak{p},r}^n \}$ are the eigenvalues for
  $\Frob_\frak{p}^n$. If we let $a_{\frak{p},n}
  (=\sum_{r=1}^{2g} \alpha^n_{\frak{p},r}) \in \Z$ denote the trace of
  $\Frob_\frak{p}^n$, we have:
\begin{equation*}
\begin{split}
a_{\frak{p},2} & = \tr (\Frob_\frak{p}^2) \equiv \sum_{r=1}^{2g}
\chi^{i_r}(\Frob_\frak{p}^2) = \sum_{r=1}^{2g} \chi^{2i_r -
  1}(\Frob_\frak{p}) \cdot \chi(\Frob_\frak{p}) \\ 
& = \sum_{r=1}^{2g} \varepsilon_{r,r}(\Frob_\frak{p}) \cdot
\chi(\Frob_\frak{p}) = 2g \chi(\Frob_\frak{p}) \equiv 2gq \pmod{\ell}.
\end{split}
\end{equation*}
On the other hand, from the Weil conjectures, we know the
$\alpha_{\frak{p},r}$ are $q$-Weil numbers, and so the trace
$a_{\frak{p},2}$ is a rational integer satisfying $|a_{\frak{p},2}|
\leq 2gq$. As $q < \frac{\ell}{4g}$, we must have $a_{\frak{p},2} =
2gq$. This forces $\alpha_{\frak{p},r}^2 = q$ for all $r$ (any other
choice of eigenvalues gives $a_{\frak{p},2} < 2gq$). Consequently:
\begin{equation*}
\alpha_{\frak{p},r} = \pm q^{1/2} = \pm p^{f_{\frak{p}/p}/2}.
\end{equation*}
Let $s^+$ and $s^-$ denote, respectively, the number of indices $r$
for which $\alpha_{\frak{p},r}$ is $+q^{1/2}$ or $-q^{1/2}$. Then the
rational integer $a_{\frak{p},1}$ satisfies
\begin{equation*}
a_{\frak{p},1} = (s^+ - s^-)p^{f_{\frak{p}/p}/2},
\end{equation*}
which is only possible (since $f_{\frak{p}/p}$ is odd) if $s^+ = s^-$
and $a_{\frak{p},1} = 0$. We now have:
\begin{equation*}
\begin{split}
0 = a_{\frak{p},1}^2 & = \left( \tr \Frob_\frak{p} \right)^2 \equiv
\left( \sum_{r=1}^{2g} \chi^{i_r} (\Frob_\frak{p}) \right)^2 
= \sum_{r=1}^{2g} \sum_{s=1}^{2g} \chi^{i_r + i_s}(\Frob_\frak{p}) \\  
& = \sum_{r,s} \varepsilon_{r,s}(\Frob_\frak{p}) \cdot
\chi(\Frob_\frak{p}) \equiv 4g^2q \pmod{\ell}.
\end{split}
\end{equation*}
Consequently, $\ell \mid 4g^2q$, or what is the same, $\ell \mid
4gq$. Clearly this contradicts $q < \frac{\ell}{4g}$, and so it must 
be that $\chi(m_\Q)(\Frob_\frak{p}) \neq 1$, as claimed.
\end{proof}
\begin{corollary}\label{cor:Mazur-over-Q}
In case $K = \Q$ we have $\chi(m_\Q)(\Frob_p) \neq 1$ for all $p <
\frac{\ell}{4g}$. 
\end{corollary}
Moreover, also in the special case $K = \Q$, we have the following
result. For any positive $\varepsilon < \frac{1}{12}$, set
\begin{equation*}
C_8 = C_8(g, \varepsilon) := \max \left\{C_7(g,1),
  C_1(2,g,\varepsilon), C_1(4, g, \varepsilon),
  (4gC_1'(2,\varepsilon)^3)^{4/(1-12\varepsilon)} \right\}.
\end{equation*}
\begin{proposition}\label{prop:no_small_m}
Suppose $0 < \varepsilon < \frac{1}{12}$ and $\ell > C_8$. If $[A] \in
\mathscr{A}(\Q,g,\ell)$, then $m_\Q > 6$.
\end{proposition}
\begin{proof}
  As $\ell > C_7(g,1)$, \eqref{A2} holds and $m_\Q$ must be even. If
  $m_\Q \leq 4$, then by Corollary \ref{cor:Elliott-corollary}, there
  is a prime $p < \frac{\ell}{4g}$ which is an $m_\Q$-th power residue
  modulo $\ell$. Thus, $\chi(m_\Q)(\Frob_p) = 1$, which contradicts
  the previous result. It only remains to eliminate the possibility
  $m_\Q = 6$. We argue by contradiction. Suppose $m_\Q = 6$, so that
  $6 \mid (\ell - 1)$. By Proposition \ref{prop:Elliott}, we know
  there exists $p < C_1'(2,\varepsilon) \cdot \ell^{1/4 +
    \varepsilon}$ such that $\chi(2)(\Frob_p) = 1$. As $\ell >
  (4gC_1'(2,\varepsilon)^3)^{4/(1-12\varepsilon)}$, we have
\begin{equation*}
p^3 < C_1'(2,\varepsilon)^3 \cdot \ell^{3/4 + 3\varepsilon} < \frac{\ell}{4g}.
\end{equation*}
A priori, the characters $\varepsilon_{r,s}$ always take values in
$\bmu_6$. However, as $p$ is a square modulo $\ell$, we must have
$\varepsilon_{r,s}(\Frob_p) \in \bmu_3$. Hence
\begin{equation*}
\begin{split}
a_{p,6} & = \tr \left( \Frob_p^6 \right) \equiv \sum_{r=1}^{2g}
\chi^{i_r}(\Frob_p^6) = \sum_{r=1}^{2g} \chi^{2i_r}(\Frob_p^3) \\ & =
\sum_{r=1}^{2g} \varepsilon_{r,r}(\Frob_p)^3 \chi(\Frob_p)^3 \equiv
2gp^3 \pmod{\ell}. 
\end{split}
\end{equation*}
From the Weil conjectures, however, we have $|a_{p,6}| \leq 2gp^3$,
and so we must have $a_{p,6} = 2gp^3$. Consequently:
\begin{equation*}
2gp^3 = \alpha_{p,1}^6 + \cdots + \alpha_{p,2g}^6.
\end{equation*}
As each $\alpha_{p,r}$ has absolute value $p^{1/2}$, we must have
$\alpha_{p,r}^6 = p^3$, hence $\alpha_{p,r} = \eta^{t_r} p^{1/2}$,
where $0 \leq t_r \leq 5$ and $\eta$ is a primitive sixth root of
unity. Thus, $\alpha_{p,r} \in \Q(\eta, \sqrt{p})$. For each $0 \leq t
\leq 5$, set $\kappa_t := \#\{r : \alpha_{p,r} = \eta^t p^{1/2}
\}$. The group $\Gal(\Q(\eta,\sqrt{p})/\Q) = \langle \sigma, \tau
\rangle$, where
\begin{equation*}
\sigma \colon \left\{ \begin{array}{rcc} \sqrt{p} & \mapsto & -\sqrt{p}
    \\ \eta & \mapsto & \eta \end{array} \right., \qquad \tau \colon
\left\{ \begin{array}{rcc} \sqrt{p} & \mapsto & \sqrt{p} \\ \eta &
    \mapsto & \eta^{-1} \end{array} \right. .
\end{equation*}
As a set (possibly with multiplicity),
$\{\alpha_{p,1},\dots,\alpha_{p,2g} \}$ is Galois stable, which yields
$\kappa_0 = \kappa_3$ and $\kappa_1 = \kappa_2 = \kappa_4 =
\kappa_5$. Moreover, 
\begin{equation*}
a_{p,j} = p^{j/2} \left( (1 + (-1)^j)\kappa_0 + (\eta^j + \eta^{2j} +
  \eta^{4j} + \eta^{5j})\kappa_1 \right),
\end{equation*}
which vanishes if $j$ is odd. So $a_{p,3} = 0$. Consequently,
\begin{equation*}
\begin{split}
0 = a_{p,3}^2 & = \bigl(\tr (\Frob_p^3) \bigr)^2 \equiv \left(
  \sum_{r=1}^{2g} \chi^{i_r}(\Frob_p^3) \right)^2 =
\sum_{r=1}^{2g} \sum_{s=1}^{2g} \chi^{i_r}(\Frob_p^3)
\chi^{i_s}(\Frob_p^3) \\
& = \sum_{r=1}^{2g} \sum_{s=1}^{2g} \varepsilon_{r,s}(\Frob_p)^3
\chi(\Frob_p)^3 = 4g^2\chi(\Frob_p)^3 \equiv 4g^2p^3 \pmod{\ell}.
\end{split}
\end{equation*}
As $\ell$ is prime, this implies $\ell \mid 2gp$. However, since $p <
p^3 < \frac{\ell}{4g}$, this is impossible; thus, $m_\Q \neq 6$.
\end{proof}

\section{Conditional Results}

In this section, we provide two proofs of the finiteness conjecture
(Conjecture \ref{conj:AKg-finite}) under the assumption of the
Generalized Riemann Hypothesis. The first proof is completely general,
in that it demonstrates the finiteness of $\mathscr{A}(K,g)$ for any
$K/\Q$. The second result is weaker, because we must add the
assumption that $n_K$ is odd. However, it is a finiteness result which
is uniform in the degree $n_K$; that is, we demonstrate the existence
of one bound $L$, dependent only on $g$ and $n_K$, but not $K$ itself,
for which $\ell > L$ implies $\mathscr{A}(K,g,\ell) = \varnothing$.

\subsection{Finiteness via Effective Chebotarev}

\begin{theorem}
  Let $K$ be a number field, and let $g > 0$. For all $\ell \gg
  0$, assume the Generalized Riemann Hypothesis holds for the Dedekind
  zeta functions of number fields of the form $L\tilde{K}$, where $L$
  is a subfield of $\Q(\bmu_\ell)$. Then $\mathscr{A}(K,g)$ is finite.
\end{theorem}
\begin{remark}
  In fact, we need only assume the Generalized Riemann Hypothesis for
  the Dedekind zeta functions of $L \tilde{K}$, where $L =
  \Q(\bmu_\ell)_m$ and $m \mid (M'(2g)n_K, \ell - 1)$. 
\end{remark}
\begin{proof}
We show $\mathscr{A}(K,g,\ell)$ is non-empty for only finitely many
$\ell$. First, let us define:
\begin{equation*}
C_9(m, g, K) := \max \{ C_6(m, g, 1, K) : m \mid \frac{1}{2}M'(2g)n_K \}
\end{equation*}
Let $\ell$ be a prime number with
\begin{equation*}
\ell > \max \{C_7(g,n_K), C_9(m, g, K)\}.
\end{equation*}
We claim $\mathscr{A}(K,g,\ell) = \varnothing$. If not, then there
exists an abelian variety $A/K$ with $[A] \in
\mathscr{A}(K,g,\ell)$. Then \eqref{A2} holds, and we define the
quantities $e$ and $m_\Q$ associated to $A$ as in \S3, \S4,
respectively. By Lemma \ref{lem:e-properties} and the observation
$m_\Q \mid \frac{e}{2}$ (\S4.1), we have $\ell > C_6(m_\Q, g, 1, K)$,
so we may apply Proposition \ref{prop:GRH-Cheb-bound} (with $L_0 =
\Q(\bmu_\ell)_{m_\Q}$ and $\sigma = 1$). Thus, there exist a rational
prime $p < \frac{\ell}{4g}$ and a prime $\frak{p} \mid p$ in $K$, for
which $f_{\frak{p}/p} = 1$, and for which $\left[
  \frac{L_0\tilde{K}/K}{\frak{p}} \right] = \{1 \}$. As $\chi(m_\Q)
\bigr|_{G_K}$ factors through $G_K \twoheadrightarrow
\Gal(L_0\tilde{K}/K)$, it follows that $\chi(m_\Q)(\Frob_\frak{p}) =
1$. On the other hand, by Proposition \ref{prop:Mazur_trick}, we
know $\chi(m_\Q)(\Frob_\frak{p}) \neq 1$, a contradiction.
\end{proof}

\subsection{A Uniform Version}

Let $F$ be a field and $n > 0$ an integer. Define the following
collection of extensions of $F$:
\begin{equation*}
\mathscr{F}(F,n) := \{ K : F \subset K, [K:F] = n \}.
\end{equation*}
\begin{conjecture}[Uniform Version]\label{conj:uniform-version}
Let $g > 0$ and $n > 0$. Then there exists a bound $N = N(g,n) >
0$ such that $\mathscr{A}(K,g,\ell) = \varnothing$ for any $K \in
\mathscr{F}(\Q,n)$ and any prime $\ell > N$. 
\end{conjecture}
We remark that the uniform version for $n=1$ is exactly equivalent to
the original finiteness conjecture for $\mathscr{A}(\Q,g)$. Thus, when
considering the uniform version, we may assume $n > 1$. In this
section, we prove the following version of the uniform conjecture:
\begin{theorem}\label{thm:GRH-uniform}
Assume the Generalized Riemann Hypothesis. Then Conjecture
\ref{conj:uniform-version} holds for any $g$ and any odd $n$.
\end{theorem}
In fact, we will prove a stronger result, of which Theorem
\ref{thm:GRH-uniform} is the specific case $F=\Q$.
\begin{theorem}
  Let $F$ be any number field, and assume the Generalized Riemann
  Hypothesis for all Dedekind zeta functions of number fields. For any
  $g > 0$ and any odd $n > 0$, there exists a bound $N = N(g,n,F)$
  such that $\mathscr{A}(K,g,\ell) = \varnothing$ for any $K \in
  \mathscr{F}(F,n)$ and any prime $\ell > N$.
\end{theorem}
\begin{remark}
  The assumption of GRH is only needed for the Dedekind zeta functions
  of number fields of the form $LK$, where $L \subseteq \Q(\bmu_\ell)$
  for some prime $\ell$, and $K/F$ is an extension of degree $n$.
\end{remark}
\begin{proof}
 Set $\mathscr{M}(g,n,F) := \{m \in \Z_{>0} : m \mid
 \tfrac{1}{2}M'(2g)n_Fn \}$. Define
\begin{equation*}
\begin{split}
N_1 = N_1(g,n,F) & := \max \{ C_6(m, g, n, F) : m \in \mathscr{M}(g,n,F)
\}, \\
N = N(g,n,F) & := \max \{N_1, C_7(g, n_F n) \}.
\end{split}
\end{equation*}
Suppose $\ell > N$ is a prime number. Let $K \in \mathscr{F}(F,n)$
(and hence $n_K = n_Fn$),
and for the sake of contradiction, suppose $[A] \in
\mathscr{A}(K,g,\ell)$. By the definition of $N$, $\ell > C_7(g,
n_K)$, and so \eqref{A2} holds. By Lemma \ref{lem:e-properties} and
\S4.1, we know the quantities $m_\Q$ and $e$ associated to $A$ satisfy
$m_\Q \mid \frac{e}{2} \mid \frac{1}{2}M'(2g)n_Fn$. Now, again by the
definition of $N$, $\ell > C_6(m_{\Q}, g, n, F)$.

Let $\tilde{F}$ denote the Galois closure of $F$ over $\Q$, and set
$\tilde{L} := \Q(\bmu_\ell)_{m_\Q}\tilde{F}$. We have assumed $\ell >
C_6(m_\Q, g, n, F)$, and so Proposition \ref{prop:GRH-Cheb-bound}
applies. Thus, there exists a rational prime $p <
(\frac{\ell}{4g})^{1/n}$ and a prime $\frak{p}_F$ of $F$ dividing $p$
for which $f_{\frak{p}_F/p} = 1$; moreover, we may assume $p$ is an
$m_\Q$-th power modulo $\ell$. Thus, ${\chi(m_\Q)(\Frob_{\frak{p}_F})=1}$.

However, as $n = [K:F]$ is odd, we may choose a prime $\frak{p} \mid
\frak{p}_F$ of $K$ such that $f_{\frak{p}/\frak{p}_F}$ is odd. Thus
$f_{\frak{p}/p} = f_{\frak{p}/\frak{p}_F} \cdot f_{\frak{p}_F/p}$ is
odd and at most $n$, and $N_{K/\Q}\frak{p} = p^{f_{\frak{p}/p}} \leq
p^n < \frac{\ell}{4g}$. Hence, by Proposition \ref{prop:Mazur_trick},
$\chi(m_\Q)(\Frob_\frak{p}) \neq 1$. But $\Frob_\frak{p} =
\Frob_{\frak{p}_F}^{f_{\frak{p}/\frak{p}_F}}$, so
$\chi(m_\Q)(\Frob_{\frak{p}_F}) \neq 1$, which gives a contradiction.
\end{proof}

\begin{remark}
  Fix an algebraic closure $\bar{\Q}$ of $\Q$. For a number field $K
  \subset \bar{\Q}$, $g > 0$, and a prime $\ell$, define
  $\bar{\mathscr{A}}(K,g,\ell)$ to be the image of
  $\mathscr{A}(K,g,\ell)$ in the set $\mathscr{A}(\bar{\Q},g)$ of
  isomorphism classes of $g$-dimensional abelian varieties over
  $\bar{\Q}$. For any $n > 0$, define 
\begin{equation*}
\mathscr{A}(n,g,\ell) := \bigcup_{K \subset \bar{\Q}, K \in
  \mathscr{F}(\Q,n)} \bar{\mathscr{A}}(K,g,\ell) \subseteq
\mathscr{A}(\bar{\Q},g). 
\end{equation*}
Conjecture \ref{conj:uniform-version} may be restated as follows:
Given $n > 0$ and $g > 0$, $\mathscr{A}(n,g,\ell) = \varnothing$ 
for $\ell$ sufficiently large. One might hope that even the set
$\mathscr{A}(n,g,\ell)$ is always finite, but this is not the case.
\end{remark}
\begin{proposition}
$\mathscr{A}(2,1,2)$ is infinite.
\end{proposition}
\begin{proof}
  For each $i \geq 0$, let $K_i \subseteq \bar{\Q}$ be the splitting
  field for $x^2 + 2^{i+1}x - 1$, and let $\epsilon_i$ denote the root
  of this polynomial given by $-2^i + \sqrt{2^{2i} +1}$. Then $[K_i :
  \Q] = 2$ for all $i$. Moreover, as the defining polynomial is monic
  with unit constant term, $\epsilon_i \in
  \mathscr{O}_{K_i}^\times$. On the other hand, $\epsilon_i - 1$ 
  satisfies $x^2 + (2^{i+1} + 2)x + 2^{i+1}$, and so $\epsilon_i - 1$
  lies in $\mathscr{O}_{K_i} \cap \mathscr{O}_{K_i}[\frac{1}{2}]^\times$.

  Let $E_i$ be the elliptic curve over $K_i$ defined by the
  equation $y^2 = x(x - 1)(x-\epsilon_i)$. Immediately we see that
  $E_i$ has good reduction away from $2$. Moreover, $E_i[2]$ is
  rational over $K_i$, and so $[E_i] \in
  \mathscr{A}(K_i,1,2)$. However, this family corresponds to
  infinitely many distinct $j$-invariants, and so the collection $\{ 
  [E_i \times_{K_i} \bar{\Q}] \}$ is an infinite subset of
  $\mathscr{A}(2,1,2)$.  
\end{proof}

\section{Ingredients from the Structure of the Special Fiber}

\subsection{Constraints from the action of inertia, I}

The aim of this section is to state and prove a formula relating the
dimension of an abelian variety to certain invariants. This will
extend the results of \cite[\S2]{Tamagawa:1995} into a more general
setting. We return to the notations of \S3. In particular, the
extension $L/K_\lambda^{\mathrm{ur}}$ corresponds to a subgroup
$J_\lambda$ of $I_\lambda$. Let $\kappa = \kappa(\lambda)$ denote the
residue field of $\mathscr{O}_{K_\lambda}$, and let $M := I_\lambda /
J_\lambda$.

Let $\mathcal{A}$ be the semistable N\'{e}ron model of $A_L :=
A_{K_\lambda} \times_{K_\lambda} L$ over $\mathscr{O}_L$. Then the
N\'eron property implies that the natural action of $M$ on $A_L$
(which is compatible with the natural faithful action on $L$) extends
to an action of $M$ on $\mathcal{A}$ (itself compatible with the
natural faithful action on $\mathscr{O}_L$). The latter action induces
a natural action of $M$ on the special fiber, here denoted
$\mathcal{A}_{\bar{\kappa}}$. It also induces a natural action on the
connected component $B := \mathcal{A}^0_{\bar{\kappa}}$ at the
origin. The actions of $M$ on $\mathcal{A}_{\bar{\kappa}}$ and $B$ are
compatible with the natural action of $M$ on $\bar{\kappa}$, which is
trivial as $M$ is a quotient of $I_\lambda$. Equivalently, $M$ acts on
$\mathcal{A}_{\bar{\kappa}}$ and $B$ \emph{over} $\bar{\kappa}$. As
$A_L$ has semistable reduction, $B$ is a semi-abelian variety over
$\bar{\kappa}$. Let $T$ denote the torus part of $B$, so that we have
the following canonical exact sequence:
\begin{equation}
\xymatrix{
0 \ar[r] & T \ar[r] & B \ar[r] & \overline{B} \ar[r] & 0
},
\end{equation}
where $\overline{B} := B/T$ is an abelian variety over
$\bar{\kappa}$. As this exact sequence is canonical, the action of $M$
preserves it. In particular, $M$ acts on $T$ and $\overline{B}$.

Let $A_{K_\lambda}^\vee$ be the dual abelian variety of
$A_{K_\lambda}$ over $K_\lambda$, and fix a polarization $\pi \colon
A_{K_\lambda} \to A_{K_\lambda}^\vee$ over $K_\lambda$. For each prime
$\ell' \neq \ell$, $\pi$ induces an isomorphism
\begin{equation}
V_{\ell'}(A_{K_\lambda}) \rightarrow V_{\ell'}(A_{K_\lambda}^\vee)
\end{equation}
of $G_{K_\lambda}$-modules. In particular, the field $L^\vee$, defined
to be the minimal Galois extension of $K_\lambda^\mathrm{ur}$ over
which $A_{K_\lambda}^\vee$ obtains semistable reduction, coincides
with $L$. Hence, the analogous quantities for the dual
$A_{K_\lambda}^\vee$ also coincide with those for
$A_{K_\lambda}$. That is, $e_{A_{K_\lambda}^\vee} =
e_{A_{K_\lambda}}$, $J_\lambda^\vee = J_\lambda$, etc. If we denote by 
$\mathcal{A}^\vee$ the semistable N\'eron model of $A_L^\vee :=
A_{K_\lambda}^\vee \times_{K_\lambda} L$ over $\mathscr{O}_L$, we similarly
obtain a canonical exact sequence
\begin{equation}
\xymatrix{
0 \ar[r] & T^\vee \ar[r] & B^\vee \ar[r] & \overline{B}^\vee \ar[r] & 0
},
\end{equation}
where $B^\vee$ is the connected component of
$\mathcal{A}^\vee_{\bar{\kappa}}$ at the origin, etc.

In the current context, we let $(\cdot)^*$ denote the functor $X
\mapsto \Hom(X, \Q_{\ell'}(1))$.
\begin{lemma}
Let $\ell' \neq \ell$ be a prime number. Then
\begin{equation}
V_{\ell'}(A_{K_\lambda})^\mathrm{ss} \cong V_{\ell'}(B) \oplus
V_{\ell'}(T^\vee)^*
\end{equation}
as $M$-modules.
\end{lemma}
\begin{proof}
Set $V := V_{\ell'}(A_{K_\lambda})$ and $V^\vee :=
V_{\ell'}(A_{K_\lambda}^\vee)$. As above, we let $\mathcal{A}^\vee$
denote the semistable N\'eron model of $A_L^\vee$ over
$\mathcal{O}_L$. Then there is a natural perfect pairing
\begin{equation}\label{eqn:pairing}
V \times V^\vee \longrightarrow \Q_{\ell'}(1).
\end{equation}
Furthermore, by \cite[Expos\'e IX]{SGA7I}, $V$ and $V^\vee$ admit the
following natural filtrations:
\begin{equation}
V \supseteq V^f \supseteq V^t \supseteq \{0 \}, \qquad \qquad V^\vee
\supseteq (V^\vee)^f \supseteq (V^\vee)^t \supseteq \{0 \}.
\end{equation}
Here, we have the following definitions/equalities:
\begin{equation}
\begin{split}
V^f & := V^{J_\lambda} = V_{\ell'}(B), \\
(V^\vee)^f & := (V^\vee)^{J_\lambda} = V_{\ell'}(B^\vee), \\
V^t & := V_{\ell'}(T), \\
(V^\vee)^t & := V_{\ell'}(T^\vee)).
\end{split}
\end{equation}
With respect to the pairing \eqref{eqn:pairing},
$V^t$ and $(V^\vee)^f$ are exact annihilators of each other. Likewise,
$V^f$ and $(V^\vee)^t$ are exact annihilators of each other. Since the
action of $J_\lambda$ on $V$ is unipotent of level $\leq 2$ (meaning that
for every $\xi \in J_\lambda$, $(\xi - \id)^2 = 0$ on $V$), we have
\begin{equation}
V^\mathrm{ss} \cong V^{J_\lambda} \oplus (V/V^{J_\lambda}) = V^f
\oplus (V/V^f) \cong V^f \oplus ( (V^\vee)^t)^*
\end{equation}
as $M$-modules, as desired. \qedhere
\end{proof}

Let $M' = \langle \gamma \rangle$ be a cyclic subgroup of $M$, and let
$e'$ be the order of $M'$. We consider the group algebras $\Z[M']$ and
$\Q[M']$ of $M'$ over $\Z$ and $\Q$, respectively. Note that we have: 
\begin{equation}
\Q[M'] \cong \Q[x]/(x^{e'} - 1) \cong \prod_{d \mid e'} \Q(\zeta_d),
\end{equation}
where the isomorphisms are given by identifying the generators
$\gamma$, $x \pmod{x^{e'}-1}$, and $(\zeta_d)_{d \mid e'}$ in each algebra. For
each divisor $d \mid e'$, let $\Phi_d(x) \in \Z[x]$ denote the $d$-th
cyclotomic polynomial. Then $\Phi_d(\gamma) \in \Z[M']$ generates
$\frak{a}_d$, the kernel of the natural homomorphism $\Z[M']
\twoheadrightarrow \Z[\zeta_d]$ defined by  $\gamma \mapsto \zeta_d$.

There is a natural $\Z$-algebra homomorphism $\Z[M'] \to \End(B)$, and
so we may consider the quotient group scheme $B_d := B/\frak{a}_d B$
of $B$, on which $\Z[M']$ acts via $\Z[M'] \twoheadrightarrow
\Z[M']/\frak{a}_d \cong \Z[\zeta_d]$. Similarly, we have homomorphisms
from $\Z[M']$ into the endomorphism rings of $T$, $B^\vee$ and
$T^\vee$, and so we may define quotient group schemes $T_d :=
T/\frak{a}_d T$, $B_d^\vee := B^\vee / \frak{a}_d B^\vee$, $T_d^\vee
:= T^\vee / \frak{a}_d T^\vee$. We observe by
\cite[Lem.~2.1(i)]{Tamagawa:1995} that $V_{\ell'}(B_d)$ and
$V_{\ell'}(T_d^\vee)$ are free $\Z[\zeta_d] \otimes_\Z
\Q_{\ell'}$-modules. Define
\begin{equation*}
\begin{split}
  b_d := & \rank_{\Z[\zeta_d] \otimes_\Z \Q_{\ell'}} V_{\ell'} (B_d), \\
  t_d := & \rank_{\Z[\zeta_d] \otimes_\Z \Q_{\ell'}} V_{\ell'} (T_d^\vee), \\
  n_d := & b_d + t_d, \\
  g_d := & \dim B_d = \dim B_d^\vee, \\
  h_d := & \dim T_d = \dim T_d^\vee.
\end{split}
\end{equation*}

\begin{proposition}\label{prop:nd_for_small_d}
Let $\varphi$ denote Euler's totient function. There are non-negative
integers $n_d$, indexed by the divisors of $e'$, such that
\begin{equation}\label{eqn:2g-relation}
2g = \sum_{d \mid e'} n_d \varphi(d), \qquad e' = \lcm \{d : n_d > 0 \}.
\end{equation}
Moreover, if $d \leq 2$, then $2 \mid n_d$.
\end{proposition}
\begin{proof} 
  The indices $n_d$ will be precisely as defined above. Note that
  $V_{\ell'}(T_d^\vee)^*$ is also a free $\Z[\zeta_d] \otimes_\Z
  \Q_{\ell'}$-module of rank $t_d$. Indeed, this follows from the fact
  that the automorphism ${\gamma \mapsto \gamma^{-1}}$ of the group
  algebra $\Z[M']$ induces an automorphism of $\Z[\zeta_d]$ (namely,
  $\zeta_d \mapsto \zeta_d^{-1}$). Further, the natural morphisms
\begin{equation}
B \rightarrow \bigoplus_{d \mid e'} B_d, \qquad T^\vee \rightarrow
\bigoplus_{d \mid e'}
T_d^\vee
\end{equation}
induce isomorphisms
\begin{equation}
\begin{split}
V^f = V_{\ell'}(B) & \cong \bigoplus_{d \mid e'} V_{\ell'}(B_d) \\
(V^\vee)^t = V_{\ell'}(T^\vee) & \cong \bigoplus_{d \mid e'} V_{\ell'}(T_d^\vee),
\end{split}
\end{equation}
respectively. The latter decomposition induces $V_{\ell'}(T^\vee)^*
\cong \bigoplus_{d \mid e'} V_{\ell'} (T_d^\vee)^*$. In summary:
\begin{equation}\label{eqn:ss-isom}
V_{\ell'}(A_{K_\lambda})^\mathrm{ss} \cong V_{\ell'}(B) \oplus
V_{\ell'}(T^\vee)^* \cong \bigoplus_{d \mid e'} \bigl( V_{\ell'}(B_d) \oplus
V_{\ell'}(T_d^\vee)^* \bigr).
\end{equation}
Of course, $\Z[\zeta_d] \otimes_\Z \Q_{\ell'}$ has dimension
$\varphi(d)$ as a $\Q_{\ell'}$-vector space, so by counting the
dimensions of $\Q_{\ell'}$-vector spaces on each side of
\eqref{eqn:ss-isom}, we obtain ${2g = \sum_{d \mid e'} n_d
  \varphi(d)}$. As $M$ acts on $V_{\ell'}(A_{K_\lambda})^\mathrm{ss}$
faithfully, we must have ${\lcm \{d : d \mid e', n_d > 0 \} = e'}$.

Next, we observe that $h_d = t_d \varphi(d)$, and also
\begin{equation*}
\begin{split}
b_d \varphi(d) & = \dim_{\Q_{\ell'}} V_{\ell'}(B_d) \\
 & = \dim_{\Q_{\ell'}} V_{\ell'}(\overline{B}_d) + \dim_{\Q_{\ell'}}
 V_{\ell'}(T_d) \\
 & = 2(g_d - h_d) + h_d = 2g_d - t_d\varphi(d).
\end{split}
\end{equation*}
Consequently $2g_d = n_d \varphi(d)$. Finally, for $1 \leq d \leq 2$,
we have $\varphi(d) = 1$, and so $2 \mid n_d$ as claimed.
\end{proof}

This allows us to improve the result of Corollary
\ref{cor:e-small-prime-factors}. 
\begin{proposition}
Let $p$ be a prime divisor of $e'$. Then
\begin{equation*}
p \leq 2 \cdot \max_{d \mid e'} \{g_d \} + 1.
\end{equation*} 
\end{proposition}
\begin{proof}
  By the lcm property satisfied by $e'$, there exists $d \mid e'$,
  $n_d > 0$, such that $p \mid d$. Then
\begin{equation*}
p = \varphi(p) + 1 \leq \varphi(d) + 1 \leq n_d \varphi(d) + 1 = 2g_d
+ 1. \qedhere 
\end{equation*}
\end{proof}
\begin{corollary}\label{cor:M-cyclic}
If $\ell > 2g + 1$, then $e_{A_{K_\lambda}}$ is prime to $\ell$, and
$M$ is cyclic.
\end{corollary}
\begin{proof}
  If $e_{A_{K_\lambda}} = \#M$ is divisible by $\ell$, then there
  exists a cyclic subgroup $M'$ of $M$ of order $e' = \ell$. By the
  previous proposition (or Corollary \ref{cor:e-small-prime-factors}),
  we then have $\ell \leq 2g + 1$, which contradicts the
  assumption. Since $e_{A_{K_\lambda}}$ is prime to $\ell$, $M$ arises
  as a tame inertia group, which is therefore cyclic. \qedhere
\end{proof}
Thus, if $\ell > 2g + 1$, we may apply the results in this section to
$M' = M$ and $e' = e_{A_{K_\lambda}}$.

\subsection{Constraints from the action of inertia, II}

Let us continue the notations of the previous subsection. However, we
now let $d$ denote a fixed divisor of $e_{A_{K_\lambda}}$, and keep
$d$ fixed throughout the current subsection. We will deduce further
constraints under the following assumptions:
\begin{equation}
\boxed{ \quad \mbox{$A_L$ has good reduction.} \quad } \tag{A3} \label{A3}
\end{equation}
\begin{equation}
\boxed{\quad \ell \nmid e_{A_{K_\lambda}}. \quad} \tag{A4} \label{A4}
\end{equation}
\begin{remark}
If we assume both \eqref{A1} and \eqref{A2}, then \eqref{A3} and
\eqref{A4} hold automatically. For \eqref{A3}, this follows from the
discussion prior to Lemma \ref{lem:e-properties}. For \eqref{A4}, this
follows from Corollary \ref{cor:e-small-prime-factors} and the
observation that $\ell > 2g+1$ under \eqref{A2}. 
\end{remark}

Under \eqref{A4}, $M$ is cyclic. Under \eqref{A3}, $B =
\mathcal{A}_{\bar{\kappa}}^0 = \mathcal{A}_{\bar{\kappa}}$; we obtain
the following decomposition and associated formula
\begin{equation*}
B \sim \!\!\!\!\!\!\!\! \bigoplus_{ \phantom{{}_{A_{K_\lambda}}
  }d \mid e_{A_{K_\lambda}}} \!\!\!\!\!\!\! B_d, 
\qquad 2g = \!\!\!\!\!\!\!\! \sum_{\phantom{ {}_{A_{K_\lambda}} }d
  \mid e_{A_{K_\lambda}}}
\!\!\!\!\!\!\!\! n_d \varphi(d).
\end{equation*}
Let $\gamma_d$ denote the $\ell$-rank of $B_d$, so that $0 \leq
\gamma_d \leq g_d$. (As we typically will assume \eqref{A1} and
\eqref{A2}, usually $\gamma_d = 0$. See the discussion prior to Lemma
\ref{lem:e-properties}.) Let $f$ and $f_\lambda$ denote,
respectively, the orders of $\ell \pmod{d}$ and
$\ell^{f_{\lambda/\ell}} \pmod{d}$ in $(\Z/d\Z)^\times$.

Let us consider the decomposition of $A_L$ in more
detail. The uniqueness of $L$, together with the fact that
$K_\lambda^\mathrm{ur}/K_\lambda$ is a Galois extension, implies that
$L/K_\lambda$ is Galois. We have the following exact sequence:
\begin{equation*}
\xymatrix{
1 \ar[r] &
\Gal(L/K_\lambda^\mathrm{ur}) \ar[r] &
\Gal(L/K_\lambda) \ar[r] &
\Gal(K_\lambda^\mathrm{ur}/K_\lambda) \ar[r] & 1
}.
\end{equation*}
As $\Gal(K_\lambda^\mathrm{ur}/K_\lambda) \cong G_{\kappa(\lambda)} \cong
\hat{\Z}$ is a free profinite group, this exact sequence
splits. So there is a group-theoretic section $s \colon
\Gal(K_\lambda^\mathrm{ur}/K_\lambda) \hookrightarrow
\Gal(L/K_\lambda)$. Let $L_0/K_\lambda$ be the subextension of
$L/K_\lambda$ corresponding to the subgroup $\mathrm{Im}(s) \subseteq
\Gal(L/K_\lambda)$. Then $L_0/K_\lambda$ is a finite, totally ramified
extension of degree $e_{A_{K_\lambda}}$, and $L = L_0 K_\lambda^\mathrm{ur}$.

In case $A_L$ has good reduction, then the inertia group $I_{L_0} \leq
G_{L_0}$ acts trivially on $V_{\ell'}(A_{K_\lambda})$ (as before,
$\ell'$ is a prime distinct from $\ell$). Hence, $A_{L_0} :=
A_{K_\lambda} \times_{K_\lambda} L_0$ has good reduction. Let
$\mathcal{A}_0$ be the proper smooth N\'eron model of $A_{L_0}$ over
$\mathscr{O}_{L_0}$, and let $B_0$ be the special fiber
$(\mathcal{A}_0)_{\kappa(\lambda)}$ of $\mathcal{A}_0$. As $A_{L_0}$
has good reduction, $B_0$ is an abelian variety over
$\kappa(\lambda)$. Necessarily, $B = B_0 \times_{\kappa(\lambda)}
\overline{\kappa(\lambda)}$.

Note that the abelian subvariety $\frak{a}_dB \subseteq B$ is stable
under the action of $G_{\kappa(\lambda)} =
\Gal(\overline{\kappa(\lambda)} / \kappa(\lambda))$. This follows from
the fact that the prime ideal $\frak{a}_d \subseteq \Z[M]$ is stable
under the natural action of $G_{\kappa(\lambda)}$ (induced by the
natural action of $G_{\kappa(\lambda)}$ on $M$). Thus, the abelian
subvariety $\frak{a}_dB \subset B$ and the quotient abelian variety $B
\twoheadrightarrow B/\frak{a}_dB = B_d$ descend to a unique abelian
subvariety and a unique quotient abelian variety $B_0
\twoheadrightarrow B_{0,d}$, respectively.

Let $\Frob_\lambda \in G_{\kappa(\lambda)}$ denote the
$\ell^{f_{\lambda/\ell}}$-th power Frobenius element of
$G_{\kappa(\lambda)}$, and let  
\begin{equation*}
F_\lambda \in \End_{\kappa(\lambda)}(B_{0,d})
\subseteq \End_{\overline{\kappa(\lambda)}} (B_d)_\Q = \End(B_d)_\Q
\end{equation*}
be the $\ell^{f_{\lambda/\ell}}$-th power Frobenius endomorphism. It
is well-known that on $V_{\ell'}(B_d)$ the actions of $F_\lambda$ and
$\Frob_\lambda$ coincide with each other. Since $\End(B_d)_\Q
\subseteq \End \bigl( V_{\ell'}(B_d) \bigr)$, this implies that the
conjugate action of $F_\lambda$ on $\End(B_d)_\Q$ induces an action of
$\Q(\zeta_d) \hookrightarrow \End(B_d)_\Q$ (provided $n_d > 0$, of
course). Necessarily, this coincides with the natural (Galois) action
of $\Frob_\lambda$ on $\Q(\zeta_d)$. Note that the latter action is
induced by $\ell^{f_{\lambda/\ell}} \pmod{d} \in (\Z/d\Z)^\times \cong
\Gal(\Q(\zeta_d)/\Q)$. In particular, we get that $F :=
F_\lambda^{f_\lambda}$ acts trivially on $\Q(\zeta_d)$. Equivalently,
$F$ and $\Q(\zeta_d)$ commute with each other in $\End(B_d)_\Q$. So
let $\Q(\zeta_d)'$ denote the centralizer of $\Q(\zeta_d)$ in
$\End(B_d)_\Q$. We have shown:
\begin{proposition}\label{prop:F-in-centralizer}
  If \eqref{A3} and \eqref{A4} hold, then $\Q(\zeta_d)[F] \subseteq
  \Q(\zeta_d)'$.
\end{proposition}
\begin{lemma}\label{lemma:Qzd'}
If $n_d = 1$, then $\Q(\zeta_d)' = \Q(\zeta_d)$.
\end{lemma}
\begin{proof}
  Certainly $\Q(\zeta_d)$ commutes with itself, and so we need only
  demonstrate that $\Q(\zeta_d)' \subseteq \Q(\zeta_d)$. However,
  $V_{\ell'}(B_d)$ is a free $\Q(\zeta_d) \otimes \Q_{\ell'}$-module
  of rank $n_d = 1$. Hence $\Q(\zeta_d)' \otimes \Q_{\ell'} \subseteq
  \Q(\zeta_d) \otimes \Q_{\ell'}$, which is only possible if
  $\Q(\zeta_d)' \subseteq \Q(\zeta_d)$. \qedhere
\end{proof}

The next proposition demonstrates the implications among
the following conditions:
\begin{quote}
\begin{enumerate}[(C1)]
\item $\gamma_d = 0$ and $g_d \leq 2$,
\item $-1 \pmod{d} \in \langle \ell \pmod{d} \rangle$ in $(\Z/d\Z)^\times$,
\item $B_d$ is supersingular (i.e., isogenous to a product of
  supersingular elliptic curves),
\item $\gamma_d = 0$,
\item $\gamma_d < g_d$ or $g_d = 0$,
\item either $2 \mid f$ or $d \leq 2$,
\item $2 \mid n_d f$,
\item $n_d f_\lambda f_{\lambda/\ell} \neq 1$.
\end{enumerate}
\end{quote}
\begin{proposition}\label{prop:C7-chain}
\begin{enumerate}[(i)]
\item Under assumptions \eqref{A3} and \eqref{A4}, the following
  implications always hold: 
\begin{equation*}
\xymatrix @R=0.33cm @C=0.4cm {
& \mathrm{(C1)} \ar@{=>}[d] & & \\
\mathrm{(C2)} \ar@{=>}[d] & \mathrm{(C3)} \ar@{=>}[r] \ar@{=>}[d] &
\mathrm{(C4)} \ar@{=>}[r] & \mathrm{(C5)} \ar@{=>}[d] \\
\mathrm{(C6)} \ar@{=>}[r]& \mathrm{(C7)} \ar@{=>}[rr]& & \mathrm{(C8)}
}
\end{equation*}
\item Moreover, if $n_d = 1$, we also have \emph{(C2)} $\Rightarrow$
  \emph{(C3)} and \emph{(C7)} $\Rightarrow$ \emph{(C6)}. That is, the
  following implications hold: 
\begin{equation*}
\xymatrix @R=0.33cm @C=0.4cm {
& \mathrm{(C1)} \ar@{=>}[d] & & \\
\mathrm{(C2)} \ar@{=>}[d] \ar@{=>}[r] & \mathrm{(C3)} \ar@{=>}[r] \ar@{=>}[d] &
\mathrm{(C4)} \ar@{=>}[r] & \mathrm{(C5)} \ar@{=>}[d] \\
\mathrm{(C6)} \ar@{<=>}[r] & \mathrm{(C7)} \ar@{=>}[rr]& & \mathrm{(C8)}
}.
\end{equation*}
\end{enumerate}
\end{proposition}
\begin{remark}
Before starting the proof, we make the following observations:
\begin{itemize}
\item If (C2) does not hold then instead of (C6) we have $2 \mid r :=
  \frac{\varphi(d)}{f}$. 
\item From the definitions, we have $f_\lambda = \frac{f}{(f,
    f_{\lambda/\ell})}$. Note that if $f_{\lambda/\ell} = 1$ and $n_d
  = 1$, then (C8) is equivalent to $d \nmid (\ell-1)$. (This occurs,
  for example, when $K = \Q$, $\lambda = (\ell)$, and $n_d = 1$.)
\end{itemize}
\end{remark}

\begin{proof}
First, let us show that (C2) implies (C6). Notice that $-1 \equiv 1
\pmod{d}$ is equivalent to $1 \leq d \leq 2$. Otherwise, $-1 \pmod{d}
\in (\Z/d\Z)^\times$ has order exactly $2$, and so $2 \mid f$. 

Apart from (C2) and (C6), every condition clearly holds if $n_d = 0$,
so we assume $n_d > 0$ for the remainder. The implications
(C1) $\Rightarrow$ (C3) $\Rightarrow$ (C4) $\Rightarrow$ (C5) are clear.

Assume (C6). If $2 \mid f$, then of course (C7) holds. If $d \leq 2$,
then by Proposition \ref{prop:nd_for_small_d}, $2 \mid n_d$ and again
(C7) holds. If $n_d = 1$, we clearly have the reverse implication. 

Assume (C7). By the remark above, note that $f_\lambda
f_{\lambda/\ell} = \frac{f_{\lambda/\ell}}{(f, f_{\lambda/\ell})} f$, and so
is divisible by $f$. Hence, $2 \mid n_d f_\lambda f_{\lambda/\ell}$, and so
(C8) holds.

Now, assume (C3). We may write $B_d \sim E^{g_d}$ for some
supersingular elliptic curve $E$. Then $\End(B_d)_\Q \cong
M_{g_d}(D)$, where $D = \End(E)_\Q$ is a quaternion algebra over $\Q$
whose set of ramified primes is exactly $\{\ell, \infty\}$. (This
determines $D$ uniquely up to isomorphism.) In particular,
$\End(B_d)_\Q$ is a central simple algebra over $\Q$ and we have
$\dim_\Q \End(B_d)_\Q = \dim_\Q M_{g_d}(D) = 4g_d^2$. As $n_d > 0$, we
have $\Q(\zeta_d) \hookrightarrow \End(B_d)_\Q$. Let $\Q(\zeta_d)'$
denote the centralizer of $\Q(\zeta_d)$ in $\End(B_d)_\Q$. Then, by
\cite[Thm.~10.7.5]{Cohn:1977}, $\Q(\zeta_d)'$ is a central simple
algebra over $\Q(\zeta_d)$, 
\begin{equation*}
\dim_{\Q(\zeta_d)} \Q(\zeta_d)' = \frac{ (2g_d)^2}{\varphi(d)^2} =
n_d^2,
\end{equation*}
and
\begin{equation*}
\End(B_d)_\Q \otimes_\Q \Q(\zeta_d) \cong M_{\varphi(d)} \bigl(
\Q(\zeta_d)' \bigr).
\end{equation*}
However, $\End(B_d)_\Q \otimes_\Q \Q(\zeta_d)$ is also isomorphic to
$M_{g_d} \bigl( D \otimes_\Q \Q(\zeta_d) \bigr)$, and so $\mbox{$D
  \otimes_\Q \Q(\zeta_d)$}$ and $\Q(\zeta_d)'$ are similar (i.e., the
elements of $\Br(\Q(\zeta_d))$ associated to these algebras
coincide). 

If $D \otimes_\Q \Q(\zeta_d)$ splits, i.e., $D \cong
M_2(\Q(\zeta_d))$, let $\mu$ be any prime of $\Q(\zeta_d)$ dividing
$\ell$. Then $f = [\Q(\zeta_d)_\mu : \Q_\ell]$, and by observing the
local invariant at $\ell$ of $[D] \in \Br(\Q)$, we deduce that $2 \mid
f$. On the other hand, if $D \otimes_\Q \Q(\zeta_d)$ does not split
(that is, $D$ remains a division algebra after tensoring with
$\Q(\zeta_d)$), then by the definition of similarity we must have $D
\otimes_\Q \Q(\zeta_d) \subseteq \Q(\zeta_d)'$, and so $2 \mid
n_d$. Thus, (C3) $\Rightarrow$ (C7).

Next, we show that (C5) $\Rightarrow$
(C8). If $n_d \neq 1$, (C8) always holds. So let us assume that $n_d =
1$. By Proposition \ref{prop:F-in-centralizer} and Lemma
\ref{lemma:Qzd'}, we may view $F$ as an element of
$\Q(\zeta_d)$. Since the characteristic polynomial of $F$ has
coefficients in $\Z$ with constant term a power of $\ell$, we see that
\begin{equation}\label{eq:Fmembership} 
F \in \Z[\zeta_d] \cap \Z[\zeta_d][\tfrac{1}{\ell}]^\times.
\end{equation}
Moreover, since the action of $F$ on $V_{\ell'}(B_d)$ is given by the
scalar action of $\Q(\zeta_d) \otimes_\Q \Q_{\ell'}$ on the rank one
free $\Q(\zeta_d) \otimes_\Q \Q_{\ell'}$-module $V_{\ell'}(B_d)$, $F$
must be an $\ell^{f_\lambda f_{\lambda/\ell}}$-Weil number (when
viewed as an element of $\Q(\zeta_d)$). Let $c \in G :=
\Gal(\Q(\zeta_d)/\Q)$ denote the restriction of complex conjugation to
$\Q(\zeta_d)$. We must have $F^c F = \ell^{f_\lambda f_{\lambda/\ell}}$.

As $\ell \nmid d$ by \eqref{A4}, the ideal $(\ell)\Z[\zeta_d]$ splits
as a product of distinct prime ideals in $\Z[\zeta_d]$. Let
$\lambda_1, \ldots, \lambda_r$ be these prime ideals, where $r =
\varphi(d)/f$. By \eqref{eq:Fmembership}, the prime ideal
factorization of the principal ideal $(F)$ must be
\begin{equation*}
(F) = \lambda_1^{m_1} \lambda_2^{m_2} \dots \lambda_r^{m_r}.
\end{equation*}
Let $D_\ell \leq G$ be the decomposition group of $\ell$ (or
equivalently, any of the $\lambda_i$). Under the isomorphism $G \cong
(\Z/d\Z)^\times$, $c$ corresponds to $-1 \pmod{d}$. Further, $D_\ell$
corresponds to the subgroup $\langle \ell \pmod{d} \rangle$. Hence,
the condition $c \in D_\ell$ is equivalent to (C2). So, if $c \in
D_\ell$, we have (C2) $\Rightarrow$ (C6) $\Rightarrow$ (C7)
$\Rightarrow$ (C8).

In case $c \not\in D_\ell$, then $r$ is even, and $\lambda_i^c \neq
\lambda_i$ for every $1 \leq i \leq r$. Relabel the ideals so that
$\lambda_{2j-1}^c = \lambda_{2j}$. Now,
\begin{equation*}
\begin{split}
(\lambda_1 \lambda_2 \cdots \lambda_r)^{f_\lambda f_{\lambda/\ell}} &
= (\ell^{f_\lambda f_{\lambda/\ell}}) \\
& = (F^c \cdot F) = (\lambda_1 \lambda_2)^{m_1 + m_2} \cdots
(\lambda_{r-1} \lambda_r)^{m_{r-1} + m_r}, 
\end{split}
\end{equation*}
and so
\begin{equation*}
f_\lambda \cdot f_{\lambda/\ell} = m_1 + m_2 = \cdots = m_{r-1} + m_r.
\end{equation*}
For the sake of contradiction, suppose (C8) does not hold; i.e.,
$f_\lambda f_{\lambda/\ell} = 1$. Then clearly $m_i = 0$ for exactly
half of the indices $i$. But (under a suitable normalization) the
$\{m_i\}$ represent the slopes of the Newton polygon of $B_d$. More
specifically, the slopes are given by:
\begin{equation*}
\underbrace{\frac{m_1}{f_\lambda f_{\lambda/\ell}}, \cdots,
    \frac{m_1}{f_\lambda f_{\lambda/\ell}}}_{\mbox{$f$ times}},
\underbrace{\frac{m_2}{f_\lambda f_{\lambda/\ell}}, \cdots,
    \frac{m_2}{f_\lambda f_{\lambda/\ell}}}_{\mbox{$f$ times}}, \cdots,
\underbrace{\frac{m_r}{f_\lambda f_{\lambda/\ell}}, \cdots,
    \frac{m_r}{f_\lambda f_{\lambda/\ell}}}_{\mbox{$f$ times}},
\end{equation*}
But this implies the Newton polygon has $f \cdot \frac{r}{2}$ slopes
of value $0$, and (as $n_d = 1$), this implies $\gamma_d = g_d > 0$,
which contradicts (C5). Thus, (C5) $\Rightarrow$ (C8).

Finally, we show, assuming $n_d = 1$, that (C2) $\Rightarrow$
(C3). Assuming (C2), we have $c \in D_\ell$, and $\lambda_i^c =
\lambda_i$ for all $i$. So $(F^c) = (F)$. We consider the prime
factorization:
\begin{equation*}
\begin{split}
(\lambda_1 \lambda_2 \cdots \lambda_r)^{f_\lambda f_{\lambda/\ell}} &
= (\ell^{f_\lambda f_{\lambda/\ell}}) \\
& = (F^c \cdot F) = (F)^2 = \lambda_1^{2m_1} \lambda_2^{2m_2} \cdots \lambda_r^{2m_r}, 
\end{split}
\end{equation*}
which gives $f_\lambda f_{\lambda/\ell} = 2m_i$ for every $i$. Again
we interpret this in terms of the Newton polygon of $B_d$, and see
that every slope is $\frac{m_i}{f_\lambda f_{\lambda/\ell}} =
\frac{1}{2}$. This is equivalent to the condition that $B_d$ is
supersingular, i.e., (C3). \qedhere
\end{proof}

\section{Unconditional Results}

In this section, we apply the results of \S6 to obtain as many
finiteness results as possible without the assumption of GRH.

\subsection{Unconditional finiteness results over $\Q$}
When $[A] \in \mathscr{A}(\Q,g,\ell)$, note that $e = e_{A_{\Q_\ell}}$ always.
Already the finiteness of $\mathscr{A}(\Q,1)$ has been established in 
\cite{Rasmussen-Tamagawa:2008}. The information coming from the
special fiber allows us to settle the conjecture over $\Q$ for $g \leq
3$.
\begin{proposition}
The set $\mathscr{A}(\Q,2)$ is finite.
\end{proposition}
\begin{proof}
  Suppose $\ell \gg 0$ and $[A] \in \mathscr{A}(\Q,2,\ell)$. By
  Proposition \ref{prop:no_small_m}, $m_\Q > 6$. But this is
  impossible under the condition $4 = \sum n_d \varphi(d)$. To see
  this, note that
\begin{equation*}
\{d : \varphi(d) \leq 4 \} = \{1, 2, 3, 4, 5, 6, 8, 10, 12 \}.
\end{equation*}
Thus, the condition $4 \mid e$ (Lemma \ref{lem:e-properties}) implies
one of $n_{12}$, $n_8$, $n_4$ must be positive. The only solution with
$n_{12} > 0$ is $4 = \varphi(12)$; in this case, $e = 12$ and $m_\Q
\mid 6$. Likewise, $n_8 > 0$ only for the solution $4 =
\varphi(8)$. In this case, $e = 8$ and $m_\Q \mid 4$. There are a
handful of equations with $n_4 > 0$; each, however, has $e = 4$ or $e
= 12$, and so $m_\Q \mid 6$. By contradiction,
$\mathscr{A}(\Q,2,\ell)$ must be empty.
\end{proof}
\begin{proposition}
The set $\mathscr{A}(\Q,3)$ is finite.
\end{proposition}
\begin{proof}
  Suppose $\ell \gg 0$ and $[A] \in \mathscr{A}(\Q,3,\ell)$. As in the
  case $g=2$, the conditions $6 = \sum n_d \varphi(d)$ and $4 \mid e$
  imply $m_\Q \leq 6$, with only four exceptions:
\begin{equation*}
\begin{split}
6 & = \varphi(8) + \varphi(6)  \\
6 & = \varphi(8) + \varphi(3)  \\
6 & = \varphi(4) + \varphi(10) \\
6 & = \varphi(4) + \varphi(5) 
\end{split}
\end{equation*}
So we may assume $A$ has a decomposition corresponding to one of the
above exceptions. For each exception, the only \emph{possible} value
for $m_\Q$ is $\frac{e}{2}$, as every other even divisor of
$\frac{e}{2}$ is at most $6$. However, in each case we observe a
choice of $d$ such that $n_d = 1$ and $d \mid \frac{e}{2}$.  Let $f$
denote the order of $\ell \pmod{d}$. We may be sure that $\gamma_d =
0$ (cf. the discussion of $\ell$-rank in \S3.3). Since $K = \Q$, we
have $f_{\lambda/\ell}f_\lambda = f$. Hence, by Proposition
\ref{prop:C7-chain}, $n_d f \neq 1$, and so $f \neq 1$. This forces $d
\nmid (\ell - 1)$, and so $d \nmid m_\Q$. Consequently, $m_\Q <
\tfrac{e}{2}$, and by contradiction, $\mathscr{A}(\Q,3,\ell) =
\varnothing$.
\end{proof}
Unconditional finiteness in the case $g=4$ is not settled, but we have
the following description of possible decompositions of the special
fiber.
\begin{proposition}
If $\ell \gg 0$ and $[A] \in \mathscr{A}(\Q,4,\ell)$, then the
decomposition of $A$ corresponds to one of the following sums, and 
$\ell$ must satisfy congruence conditions as follows:
\end{proposition}
\begin{figure}[h!]
\begin{center}
\begin{tabular}{rcl}
Sum & & Congruence \\
\hline \hline
$2 \varphi(3) + \varphi(8)$ & & $\ell \equiv 13 \pmod{24}$ \\
$2 \varphi(6) + \varphi(8)$ & & $\ell \equiv 13 \pmod{24}$ \\
$\varphi(16)$ & & $\ell \equiv 9 \pmod{16}$ \\
$\varphi(20)$ & & $\ell \equiv 11 \pmod{20}$ \\
$\varphi(24)$ & & $\ell \equiv 13 \pmod{24}$ 
\end{tabular}
\end{center}
\end{figure}

\begin{proof}
  Take $\ell \gg 0$ and $[A] \in \mathscr{A}(\Q,4,\ell)$. If the
  decomposition of $A$ does not correspond to one of those given in
  the table, then by arguments similar to the previous cases, one may
  show that $m_\Q \leq 6$. However, the five cases above remain
  valid. For example, in case $n_3 = 2$ and $n_8 = 1$, we have $e =
  24$, and the available results do not eliminate the possibility
  $m_\Q = 12$. We may conclude only that $\ell \equiv 1 \pmod{12}$ and
  $\ell \not\equiv 1 \pmod{8}$, i.e., $\ell \equiv 13 \pmod{24}$. The
  congruences in the remaining cases may be deduced similarly.
\end{proof}

\subsection{Finiteness of $\mathscr{A}(K,1)$ when $n_K = 2$}

In the authors' earlier work, it was shown that $\mathscr{A}(K,1)$
must be finite when $K/\Q$ is a quadratic extension, except possibly
when $K$ is imaginary with class number one
\cite[Thm.~4]{Rasmussen-Tamagawa:2008}. There, the proof uses the
result of Momose \cite{Momose:1995} (generalizing the previous work of
Mazur \cite{Mazur:1978}) classifying $K$-rational points on modular
curves. We give a different proof, which removes the exception
for imaginary fields with class number one. The new proof relies
on the theorem of Goldfeld from \S2. 
\begin{proposition}\label{prop:quad-case}
Suppose $n_K = 2$. Then the set $\mathscr{A}(K,1)$ is finite.
\end{proposition}
\begin{proof}
  If $\mathscr{A}(K,1)$ is infinite, then we may choose $\ell$ and
  $[A]$ so that $\ell > C_2(K)$, $\ell \nmid \Delta_K$, and $[A] \in
  \mathscr{A}(K,1,\ell)$. The only solution to $2 = \sum n_d
  \varphi(d)$ for which $4 \mid e$ is $2 = \varphi(4)$. So in fact $e
  = 4$ and $m_\Q = 2$. By Corollary \ref{cor:Goldfeld}, there exists
  $p < \frac{\ell}{4}$, which is a square residue modulo
  $\ell$. Moreover, $f_{\frak{p}/p} = 1$ for any prime $\frak{p}$ of
  $K$ above $p$. Consequently, $\chi(2)(\Frob_\frak{p}) =
  \chi(2)(\Frob_p) = 1$. However, by Proposition
  \ref{prop:Mazur_trick}, $\chi(2)(\Frob_\frak{p}) \neq 1$. By
  contradiction, $\mathscr{A}(K,1)$ is finite.
\end{proof}

\begin{remark}
Unfortunately, since Goldfeld's result is not effective, Proposition
\ref{prop:quad-case} is not effective, even for a particular choice of
quadratic field $K$, and cannot be made uniform at present.
\end{remark}

\subsection{Additional finiteness results when $g=1$}

We present two more unconditional finiteness results. The first
establishes finiteness for cubic fields in a uniform manner; that is,
there exists a bound $N$ such that $\ell > N$ implies
$\mathscr{A}(K,1,\ell) = \varnothing$ for any cubic field $K$. The
second result is not uniform, but provides finiteness for
$\mathscr{A}(K,1)$ for any Galois extension $K/\Q$ whose Galois group
has exponent $3$.

We begin with the result for cubic fields. We require an extension of
the result of Proposition \ref{prop:Mazur_trick}. Let $K/\Q$ be a
finite extension, and suppose $[A] \in \mathscr{A}(K,1,\ell)$. Note
that as \eqref{A2} holds for $\ell > C_7$, which depends only on $g$
and $n_K$, we may use the results of \S5 without invalidating
uniformity.

Provided $\ell > 3$, for each prime $\lambda \mid \ell$ in $K$, we
have a decomposition of $V_{\ell'}(A_{K_\lambda})^\mathrm{ss}$ which
yields 
\begin{equation*}
2 = 2g = \!\!\!\!\!\!\!\! \sum_{\phantom{ {}_{A_{K_\lambda}} }d \mid  e_{A_{K_\lambda}}}
\!\!\!\!\!\!\!\! n_d \varphi(d), \qquad e_{A_{K_\lambda}} = \lcm \{ d
: n_d > 0 \}.
\end{equation*}
(Note that the collection $\{n_d \}$ depends on $\lambda$, even though
this is suppressed in the notation.) Necessarily, $e_{A_{K_\lambda}}
\in \{1, 2, 3, 4, 6 \}$.
\begin{proposition}
Conjecture \ref{conj:uniform-version} holds in case $(g,d) = (1,3)$.
\end{proposition}
\begin{proof}
  In this case, we take $\ell > C_7(1,3) > 3$.  so that the
  above relations hold, and that further $4 \mid e := \gcd
  \{e_{A_{K_\lambda}} e_{\lambda/\ell} : \lambda \mid \ell
  \}$. Additionally, we have $m_\Q \mid (\frac{e}{2}, \ell - 1)$. As
  $K$ is a cubic extension, we have $\sum_{\lambda \mid \ell}
  e_{\lambda/\ell} f_{\lambda/\ell} = n_K = 3$. If there exists a
  prime $\lambda \mid \ell$ for which $e_{\lambda/\ell} = 1$, then as $4
 \mid e_{A_{K_\lambda}} e_{\lambda/\ell}$, we must have
  $e_{A_{K_\lambda}} = 4$, and consequently $e = 4$. The only other
  possibility is that there is a unique prime $\lambda \mid \ell$, for
  which $e_{\lambda/\ell} = 3$. In this case, we have
  $e_{A_{K_\lambda}} = 4$, $e = 12$. Consequently, we may be sure that
  $m_\Q = 2$ or $m_\Q = 6$.

  Choose $0 < \varepsilon < \frac{1}{12}$. By Proposition
  \ref{prop:Elliott}, there exists $p =
  O(\ell^{\frac{1}{4}+\varepsilon})$ such that $\chi(2)(\Frob_p) =
  1$. As $n_K = 3$, we may always choose $\frak{p} \mid p$ in $K$ such
  that $f_{\frak{p}/p} \in \{1, 3 \}$. Let $q =
  \#\kappa(\frak{p})$. We have $q \leq p^3 =
  O(\ell^{\frac{3}{4}+3\varepsilon})$. There is an absolute constant
  $C_{10}$, independent of $K$, for which $\ell > C_{10}$ guarantees
  $q < \frac{\ell}{4}$. Below, we give a mild extension of Proposition
  \ref{prop:Mazur_trick}, showing that $\chi(2)(\Frob_\frak{p}) \neq
  1$. But this is a contradiction, since $\chi(2)(\Frob_\frak{p}) =
  \chi(2)(\Frob_p)^{f_{\frak{p}/p}} = 1$.
\end{proof}
\begin{proposition}
Suppose $n_K = 3$, $\ell > C_7(1,3)$, and $[A] \in
\mathscr{A}(K,1,\ell)$. Let $p$ be a rational prime, and $\frak{p}
\mid p$ a prime in $K$. If $f_{\frak{p}/p}$ is odd and $q =
\#\kappa(\frak{p}) < \frac{\ell}{4}$, then $\chi(2)(\Frob_\frak{p})
\neq 1$. 
\end{proposition}
\begin{proof}
By Proposition \ref{prop:Mazur_trick}, we already have
$\chi(m_\Q)(\Frob_\frak{p}) \neq 1$. If $m_\Q = 2$, we are
done. Otherwise, we have $\chi(6)(\Frob_\frak{p}) \neq 1$. For the
sake of contradiction, suppose that $\chi(2)(\Frob_\frak{p}) =
1$. Then $\chi(6)(\Frob_\frak{p})$ has exact order $3$; so 
$\varepsilon_{r,r}(\Frob_\frak{p})$ must be exactly of order $3$ for some
$r \in \{1,2\}$. Since the determinant of $\rho_{A,\ell}$ must be
$\chi$, we have $\chi^{i_1}\chi^{i_2} = \chi$; that is, $i_2 \equiv 1 -
i_1 \pmod{(\ell-1)}$. Hence,
\begin{equation*}
\varepsilon_{1,1}\varepsilon_{2,2} = \chi^{2i_1 - 1}\chi^{2i_2 - 1} =
1.
\end{equation*}
Let $\omega = \varepsilon_{1,1}(\Frob_\frak{p}) \in
\F_\ell^\times$. It is necessarily a primitive cube root of unity in
$\F_\ell$, and so $\varepsilon_{2,2} (\Frob_\frak{p}) =
\omega^2$. Computing the trace of $\Frob_\frak{p}^2$, we have
\begin{equation*}
\begin{split}
a_{\frak{p},2} & = \chi^{i_1}(\Frob_\frak{p}^2) +
\chi^{i_2}(\Frob_\frak{p}^2) \\
& = \left( \varepsilon_{1,1}(\Frob_\frak{p}) +
  \varepsilon_{2,2}(\Frob_\frak{p}) \right)\chi(\Frob_\frak{p}) \\
& \equiv -q \pmod{\ell}.
\end{split}
\end{equation*}
As $|a_{\frak{p},2}| \leq 2q$ by the Weil conjectures and $q <
\frac{\ell}{4}$, we have $a_{\frak{p},2} = -q$.Thus, the
eigenvalues $\alpha_{\frak{p},j}^2$ (in $\bar{\Q}$) for
$\Frob_\frak{p}^2$ are roots of $T^2 + qT + q^2$, hence are $\{\zeta
q, \zeta^2 q \}$, where $\zeta$ denotes a primitive cube root of unity
in $\bar{\Q}$. Consequently, the eigenvalues for $\Frob_\frak{p}$ over
$\bar{\Q}$ are contained in $\{ \pm \zeta\sqrt{q}, \pm \zeta^2\sqrt{q}
\}$, and each of these possible eigenvalues has a minimal polynomial of
degree $4$ over $\Q$. But this is absurd; the eigenvalues should have
a minimal polynomial of degree $2$ over $\Q$. By contradiction,
$\chi(2)(\Frob_\frak{p}) \neq 1$.
\end{proof}
We now turn to the result on Galois extensions of exponent $3$.
\begin{proposition}
Suppose $K/\Q$ is a Galois extension and $\Gal(K/\Q)$ has exponent
$3$. Then the set $\mathscr{A}(K,1)$ is finite.
\end{proposition}
\begin{proof}
  Take $\ell \gg 0$. Suppose $[A] \in
  \mathscr{A}(K,1,\ell)$. By avoiding the primes dividing $\Delta_K$,
  we know the decomposition $2g = \sum n_d \varphi(d)$ corresponds to
  $n_4 = 1$, $e = 4$, and $m_\Q = 2$. 

  Let $0 < \varepsilon < \frac{1}{12}$. From Proposition \ref{prop:Elliott},
  we know there exists a prime $p = O(\ell^{\frac{1}{4}+\varepsilon})$
  which is a square modulo $\ell$; hence, $\chi(2)(\Frob_p) = 1$. Let
  $\frak{p}$ be a prime in $K$ above $p$ of norm $q$. Because
  $\Gal(K/\Q)$ is exponent $3$, we must have $f_{\frak{p}/p} \in \{1,
  3\}$. In particular $f_{\frak{p}/p}$ is odd, and $q \leq p^3 =
  O(\ell^{\frac{3}{4}+3\varepsilon}) < \frac{\ell}{4g}$ if we take
  $\ell$ sufficiently large. By Proposition \ref{prop:Mazur_trick}, we
  also have $\chi(2)(\Frob_\frak{p}) \neq 1$. This is a
  contradiction, since $\Frob_\frak{p} = \Frob_p^{f_{\frak{p}/p}}$. By
  contradiction, we have that $\mathscr{A}(K,1,\ell)$ is empty for all
  large $\ell$, and  so $\mathscr{A}(K,1)$ is finite.  
\end{proof}

\section*{Appendix}

In \S2, we noted the existence of a constant $C_6$, such that $\ell >
C_6$ implies
\begin{equation}\label{eq:C6inequality}
\left( \frac{\ell}{4g} \right)^{\frac{1}{n}} > C_3 \cdot (C_4 +
C_5 \log \ell)^2. 
\end{equation}
We now give a brief derivation for a choice of $C_6$. The principal
tool will be the Lambert $W$-function; this is a multivalued complex
function defined as follows: for every $z \in \C$, $W(z)$ is a
solution to the equation $W(z)\exp W(z) = z$. The basic properties of
$W(z)$ that we will need are all thoroughly explained in
\cite{Corless:1996}.

If we restrict our consideration to real values, there are only two
branches of $W$ of interest, which we denote $W_0(x)$ and
$W_{-1}(x)$. The real function $W_0$ is an increasing function defined
on $[-e^{-1},\infty)$, and the real function $W_{-1}$ is a decreasing
function defined on $[-e^{-1},0)$. We have $-1 = W_0(-e^{-1}) =
W_{-1}(-e^{-1})$; there are no real solutions for $x<-e^{-1}$. (All of
this may be observed by na\"ively `inverting' the real function $x(w) =
we^w$.)

\begin{lemma*}
  Suppose $c$ and $N$ are positive constants, and $c \geq
  (\frac{e}{N})^N$. Then the largest real solution to the equation 
$x^\frac{1}{N} = \log(cx)$ is given by
\begin{equation*}
  x_0 = x_0(c,N) = \frac{1}{c} \exp \left( -NW_{-1}(-\frac{1}{N} 
c^{-\frac{1}{N}}) \right).  
\end{equation*}
\end{lemma*}
\begin{proof}
From $x^\frac{1}{N} = \log(cx)$, we have
\begin{equation*}
\begin{split}
  -\frac{1}{N} c^{-\frac{1}{N}} \cdot (cx)^\frac{1}{N} & =
  -\frac{1}{N} \log (cx), \\
  -\frac{1}{N} c^{-\frac{1}{N}} & = -\frac{1}{N} \log (cx) \cdot \exp
  \left( -\frac{1}{N} \log (cx) \right).
\end{split}
\end{equation*}
The assumptions guarantee $-\frac{1}{N} c^{-\frac{1}{N}} \in
[-e^{-1},0)$, so by the definition of $W_j$, we see
$-\frac{1}{N} \log (cx) = W_j(-\frac{1}{N}c^{-\frac{1}{N}})$ for $j
\in \{-1, 0 \}$. Solving for $x$, we find
\begin{equation*}
x = \frac{1}{c} \exp \left( -NW_j(-\frac{1}{N}c^{-\frac{1}{N}}) \right).
\end{equation*}
That the larger solution corresponds to the index $j=-1$ follows from
the fact that ${W_{-1} \leq W_0 < 0}$ on $[-e^{-1},0)$.
\end{proof}
For the remainder, we write $W$ for the specific branch $W_{-1}$. At
any point $x \in (-e^{-1},0)$, differentiating $W(x)e^{W(x)} = x$
yields
\begin{equation*}
W'(x) = \frac{e^{-W(x)}}{1+W(x)} = \frac{1}{x} \cdot \frac{W(x)}{1 + W(x)}.
\end{equation*}
As $W$ is a decreasing function and $y/(1+y)$ is increasing for all $y
\neq -1$, we have for any $x \in (-\frac{1}{4},0)$:
\begin{equation*}
\frac{W(x)}{1 + W(x)} < \frac{W(-\tfrac{1}{4})}{1 + W(-\tfrac{1}{4})} 
\approx 1.867\cdots < 2.
\end{equation*}
We define
\begin{equation*}
  L(x) = \begin{cases} W(-\tfrac{1}{4}) & x \in [-e^{-1},
    -\tfrac{1}{4}) \\
 2 \log(-x) & x \in [-\tfrac{1}{4}, 0) \end{cases}.
\end{equation*}
\begin{lemma*}
For all $x \in [-e^{-1},0)$, $-1 \geq W(x) \geq L(x)$.
\end{lemma*}
\begin{proof}
  As $W$ is decreasing, the inequality is clear for $x \leq
  -\frac{1}{4}$. Now, set $f(x) = W(x) - L(x)$, and observe
  $f(-\frac{1}{4}) > 0$. For $x > -\frac{1}{4}$, we have
\begin{equation*}
f'(x) = \frac{1}{x} \cdot \left(\frac{W(x)}{1+W(x)} - 2\right) > 0,
\end{equation*}
since $x < 0$ and $W(x)/(1+W(x)) < 2$. So $W(x) \geq L(x)$, as claimed.
\end{proof}
\begin{lemma*}
Suppose $c \geq (\frac{4}{N})^N$. Then $x_0 \leq c \cdot N^{2N}$.
\end{lemma*}
\begin{proof}
  We have $-\frac{1}{N} c^{-\frac{1}{N}} > -\frac{1}{4}$, and so, combining
  the previous Lemmas,
\begin{equation*}
  x_0 \leq \frac{1}{c} \exp \left(-N \cdot 2 \log \left( \frac{1}{N} 
c^{-\frac{1}{N}} \right) \right) = c \cdot N^{2N}. \qedhere
\end{equation*}
\end{proof}
\begin{corollary*}\label{cor:C6bound}
If a positive integer $\ell$ satisfies \eqref{eq:C6inequality}, then
\begin{equation*}
\ell \leq 16g^2C_3^{2n}C_5^{4n}(2n)^{4n}\exp \left( \frac{C_4}{C_5} \right).
\end{equation*}
\end{corollary*}
\begin{proof}
We rewrite the inequality \eqref{eq:C6inequality} as
\begin{equation*}
  x^{\frac{1}{N}} > \log (cx), \qquad 
  x = \frac{\ell}{4gC_3^nC_5^{2n}}, \quad N = 2n, \quad
  c = 4gC_3^nC_5^{2n} \exp \left( \frac{C_4}{C_5} \right).
\end{equation*}
We observe that $c \geq 4 \geq (\frac{4}{N})^N$. Thus, $x \leq x_0
\leq c \cdot N^{2N}$, and the claim follows immediately.
\end{proof}

\bibliographystyle{alpha}
\bibliography{Ras-Tam-bib}

\end{CJK}
\end{document}